\documentclass[12pt,leqno]{amsart}
\usepackage{amssymb,amsmath,amsthm}
\usepackage{graphicx}
\usepackage{color}
\definecolor{refkey}{gray}{.75}
\def\e{{\rm e}}
\def\eps{\varepsilon}

\def\d{{\rm d}}
\def\dist{{\rm dist}}

\def\ct{\mathbf{t}}
\def\Hh {{\mathsf H}}

        \def\sh {{\mathsf{sh}}}

\def\t {{\mathbf T}}

\def\R {\mathbb{R}}

\def\supp {{\mathrm{supp}}}
\def\V {{\mathbf V}}

\def\G{{\mathcal G}}

\def\C {{\mathcal C}}
\def\L{\mathbf{L}}
\def\O{\mathbf{O}}
\def\D {{\mathcal D}}
\def\E {{\mathbb E}}

\def\I {{\mathcal I}}
\def\RR {{\mathcal R}}
\def\F {{\mathcal F}}
\def\S {{\mathbf S}}

\def \ecc{\mathsf{ecc}}

\def\size {{\mathrm{size}}}

\def\T {{\mathcal T}}

\def\M {{\mathsf M}}

\def\Z {{\mathbb Z}}
\def\1 {{\mbox{\boldmath 1}}}

\def \l {\langle}
\def \r {\rangle}

\def \ann{\mathsf{ann}}
\def \annf{\mathrm{ann}}
\def \scl{\mathsf{scl}}

\def \and{\quad\text{and}\quad}

\newcommand{\cic}[1]{\mbox{\boldmath$#1$}}

\def \no#1#2#3 {{\bf #1} (#3), #2.}
\def \eds#1#2#3 {#1, #2, #3.}

\newtheorem{proposition}{Proposition}[section]

\newtheorem{theorem}{Theorem}
\newtheorem*{cj}{Conjecture}

\newtheorem{lemma}[proposition]{Lemma}

\theoremstyle{definition}
\newtheorem{definition}[proposition]{Definition}
\newtheorem{remark}[proposition]{Remark}
\newtheorem*{remark*}{Remark}
\newtheorem*{warn*}{A word of warning}

\numberwithin{equation}{section}

\title[ $L^p$ bounds for maximal directional singular integrals]
{Logarithmic $L^p$ bounds for maximal directional singular integrals in the plane}
\author[C.\ Demeter]
{Ciprian Demeter}
\address{
Dept.\ of Mathematics
\newline\indent
Indiana University
\newline\indent
Bloomington, IN 47405 - USA}
\email{demeter@indiana.edu {\rm (C.\ Demeter)} }

\author[F.\ Di Plinio]
{Francesco Di Plinio}
\address{
Dept.\ of Mathematics  \ \& \ Institute for Scientific Computing and Applied Mathematics
\newline\indent
Indiana University
\newline\indent
Bloomington, IN 47405 - USA}
\email{fradipli@indiana.edu {\rm (F.\ Di Plinio)} }

\subjclass{Primary: 42B20; Secondary: 42B25.}
 \keywords{Maximal singular integrals, Kakeya maximal function}
\thanks{The first author is partially supported by a Sloan Research Fellowship and by NSF Grant DMS-0901208. The second author was partially
supported by the National Science Foundation under the grant
  NSF-DMS-0906440, and by the Research Fund of Indiana University.}
\begin{document}

\begin{abstract}  Let $K$ be a Calderon-Zygmund convolution kernel on $\R$. We discuss the $L^p$-boundedness of the maximal directional singular integral
$$
T_{\V} f (x)= \sup_{v \in \V} \Big| \int_\R f(x+t v) K(t) \, \d t \Big|
$$
where $\V$ is a finite set of $N$ directions. Logarithmic  bounds (for $2\leq p<\infty)$ are established for a set $\V$ of arbitrary structure.
Sharp bounds are proved for lacunary   and   Vargas sets of directions.   The latter include the case of uniformly distributed directions and the finite truncations of the Cantor set.

We make use of both classical harmonic analysis methods and  product-BMO based time-frequency analysis techniques. As a further application of the latter, we derive an $L^p$ almost orthogonality principle for Fourier restrictions to cones.
\end{abstract}
\maketitle
\noindent
\section{Introduction and main results}
\subsection{Maximal directional singular integrals} Let $m$ be a H\"ormander-Mikhlin multiplier on $\R$.
For $f \in \C^\infty_0(\R^2)$, and $v \in S^1$, consider the directional multiplier
\begin{equation}
\label{sect1-Tmv}
T_{v} f (x_1,x_2)= \int_{\R^2} \hat f(\xi_1,\xi_2) m( \xi \cdot v)\e^{i(x_1\xi_1 +x_2\xi_2)} \, \d \xi_1 \d \xi_2.
\end{equation}
The multiplier $T_v$ admits the singular integral  representation
$$
T_{v} f (x) = \int_\R f(x+t v) K(t) \, \d t, \qquad x \not\in  \mathrm{supp}\, f,
$$
where  $K=\check m$ is a Calderon-Zygmund kernel. The most basic example   is the directional Hilbert transform
$$
H_{v} f (x) = \int_\R f(x+t v)  \, \frac{\d t}{t}, \qquad x \not\in \mathrm{supp}\, f,
$$
which corresponds to the multiplier $\mathrm{sign}(\xi)$. Assume now that we have a vector field $\textbf{v}:\R^2\to S^1$ and define the corresponding directional singular integral operator
$$T_{\textbf{v}}(f)(x)=T_{\textbf{v}(x)} f(x).$$
Various cases where the vector field   satisfies certain geometric or analytic constraints have been recently investigated. In \cite{LL1} and \cite{LL2}, Lacey and Li developed a beautiful theory for $C^{1+\epsilon}$ vector fields. Their analysis revealed intricate connections between the directional Hilbert transform, Carleson's theorem on the convergence of Fourier series and a certain restricted version of the Kakeya maximal function.

The Lacey-Li techniques have been recently refined by Bateman and Thiele. They combined  novel density estimates for rectangles with a new approach to proving vector valued inequalities, to obtain $L^p$ estimates for $H_{\textbf{v}}$ in the case when the vector field depends on just one coordinate \cite{BaTh}. On the other hand, Stein and Street \cite{SS1} have obtained analogous results for the case of real analytic vector fields.

Continuing the initial investigation from \cite{CD10}, in this paper we analyze another case of interest, the one where the range of the vector field is finite.   More precisely, we are concerned with  $L^p$ bounds for the maximal directional multiplier
\begin{equation}
\label{sect1-mdm}
T_{\V} f (x)= \sup_{v \in \V} \big| T_v f(x) \big|,\end{equation}
where $\V \subset S^1$ is a set of $N$ directions. Due to the special role played by the maximal directional Hilbert transform, we will reserve for it the notation $H_\V$.

It has been long recognized that there is a close connection between maximal and singular integral operators in harmonic analysis, and indeed, in most classical cases the two types of operators have similar boundedness properties. In many ways this paper will argue against the generality of this principle. We start by briefly recalling some history. The relevant directional maximal function
\begin{equation}
\label{sect1-Mv}
M_{\V} f (x)= \sup_{v \in \V} \big| M_v f(x) \big|\qquad M_{v} f (x)=\sup_{\eps>0}\frac{1}{2\eps} \Big| \int_{-\eps}^\eps   f(x+tv) \, \d t \Big|,
\end{equation}
has been intensely studied and is now completely understood. The sharp bounds (for generic $\V$)
\begin{equation}
\label{sect1-Mvbound}
\|M_\V\|_{2\to 2} \lesssim {\log N}, \quad \|M_\V\|_{2\to 2,\infty} \lesssim \sqrt{\log N}, \quad \|M_\V\|_{p\to p}  \lesssim (\log N)^{\frac1p}, \quad p>2,
\end{equation}
were first proved  in \cite{KATZ1}. The more recent work \cite{BAT} shows that bounds become independent of $N$
$$\|M_\V\|_{p\to p} \lesssim_p 1,\qquad 1<p<\infty,$$
if and only if $\V$ is lacunary of finite order.

In striking contrast to $M_{\V}$, the sharp (in terms of $N$) upper bounds for the operator norms $\|H_\V\|_{p\to p}$ are only known for $1<p\le 2$.
Taking $f$ to be the indicator function of the unit ball and $\V$ uniformly distributed, one gets the lower bounds (see \cite{CD10})
$$
\|H_{\V} f \|_{p } \gtrsim N^{\frac1p}\|f\|_p, \; 1<p<2, \qquad \|H_{\V} f \|_{2 } \gtrsim \log N\|f\|_2.
$$
These lower bounds are in fact also upper bounds for generic $\V$. The case $p=2$ is a simple consequence of the Rademacher-Menshov theorem. Indeed, $|H_{\V}f(x)|$ is bounded by the sum of two operators of the form
\begin{equation}
\label{introRM}
\max_{1\leq \nu \leq N}\Big| \sum_{j=1}^{\nu} f_j (x) \Big|
\end{equation}
where $f_j$ are  Fourier restrictions of $f$ to disjoint frequency cones rooted at the origin. On the other hand, for each $N$ there exists an orthonormal set $\{f_j: j=1, \ldots, N\}$ with    $$
\Big| \Big\{ x \in \R^2: \max_{1\leq \nu \leq N}\Big| \sum_{j=1}^{\nu} f_j (x) \Big|\gtrsim \sqrt{N}\log N \Big\} \Big|
> 1
$$
(this is Menshov's counterexample, see \cite{OS}); thus one cannot get a weak $L^2$ bound for $H_\V$ which is better than the strong $L^2$ bound  by simply invoking Rademacher-Menshov type results.

 In the general case, the  sharp bound
$$
\|T_\V f\|_2 \lesssim \log N \|f\|_2
$$
 has been established in \cite{CD10}. Quite surprisingly, the article \cite{KARAG} proves the lower bound  \begin{equation} \label{karag}
\|H_\V \|_{2 \to 2,\infty} \gtrsim \sqrt{\log N}, \qquad \forall\; \V \subset S^1, \#\V=N.
\end{equation}Thus $H_\V$ is an unbounded operator on $L^2$  as soon as $\V$ contains infinitely many directions, in particular for infinite lacunary sets. This is in sharp  contrast with the directional maximal function.

Our main goal is to investigate the sharp dependence on $N$ of the $L^p$ ($2<p<\infty$) and also weak $L^2$ operator norms of $T_\V$.   In  analogy with the case of the maximal function \eqref{sect1-Mvbound}, one expects logarithmic-type bounds for all $2\leq p <\infty$. However, at least three structural differences between $T_\V$ and $M_\V$ make the former harder to treat. The maximal directional singular integral is not a positive operator; it is a {\em sum} not a {\em maximum} of single scale operators.  Moreover, unlike in the case of $M_\V$ (which is trivially bounded on $L^\infty$ for any $\V$), there is no immediate evidence of an endpoint result for $T_\V$.

We briefly mention some of the  preexisting literature on maximal directional singular integrals. In \cite{CDR},   the authors deal with the boundedness of $\{H_v f(x):x\in \R^n,v \in S^{n-1}\}$ and of the relative maximal truncations   in the mixed norm spaces $L^p(\R^n; L^q(S^{n-1}))$, $1<p,q<\infty$.
In the article \cite{DV}, it is proven that,  whenever $\V\subset S^1$ is a set with Minkowski dimension $d(\V)<1$,   $M_\V$ and $H_\V$ are bounded operators on the subspace of the radial functions in $ L^p(\R^2)$ for $p>1+d(\V)$, and unbounded for $p<1+d(\V)$. The more recent article \cite{KIM} reproves that $\|H_\V\|_{2 \to 2}\sim \log N$ when $\#\V=N$ without appealing to Rademacher-Menshov explicitly, and by the same method obtains the sharp $L^2$ bound (which is $O(N^{\frac n2-1})$) for the analogous maximal directional Hilbert transform in $\R^ n$, $n>2$. Our results partially answer both open problems mentioned in \cite{KIM}.

\subsection{Notation} Throughout the paper, $m$ stands for a H\"ormander-Mikhlin multiplier on $\R$, that is
$$
m \in \C^{\infty}(\R\backslash\{0\}), \qquad |\partial^{\alpha} m(\xi)| \lesssim_\alpha |\xi|^{-\alpha}, \;  \alpha \geq 0;
$$
with corresponding multiplier operator
$$
T: L^2 (\R) \to L^2 (\R), \qquad Tf(x) = \int_\R \hat f(\xi) m(\xi) \e^{ix \xi}\, \d \xi.
$$
We will also make use of the maximal truncation  of $T$, and adopt the notation
$$
T_\star f( x) = \sup_{\eps>0}|T_\eps f(x)|, \qquad T_\eps f(x)=  \int_{|t|>\eps} f(x-t) K(t) \, \d t , \qquad K=\check m.
$$
The values of the positive constants $C,c$, as well as the implicit constants hidden in the notation $\lesssim$ may vary from line to line, and, unless otherwise specified, are  allowed to depend only on the multiplier $m$.

The Hardy-Littlewood maximal function of $f$ on $\R^n$ ($n=1,2$) is denoted by $Mf$.
For    $R \subset \R^1$ or $\R^2$, we set $\mathbb E_R f= \frac{1}{|R|}\int_R f,$
and, for a collection $\RR$ of rectangles in the plane, we denote the corresponding maximal function by
$$
\M_\RR f(x)= \sup_{R \in \RR} |\mathbb E_ R f| \cic{1}_R (x), \qquad x \in \R^2.
$$

\subsection{The main results}
The first main result concerns arbitrary finite sets of directions. We believe that the exponent one of the term $\log N$  gets sharp in the limiting case $p\to\infty,$ see Conjecture \ref{general_conj}.

\begin{theorem} \label{main-ub1}
Let $m$ be a H\"ormander-Mihlin multiplier and let be $T_\V$ defined as in \eqref{sect1-mdm}. For  any given set $\V$ of $N$ directions
\begin{equation}
\label{main-ub1eq}
\|T_\V \|_{p\to p} \leq Cp \log N, \qquad 2<p<\infty.
\end{equation}
Moreover the following endpoint result holds: if $\supp \,f \subset Q \subset \R^2$
\begin{equation}
\label{main-ub1ep}
\big|\big\{x \in Q: |T_\V f(x)|> \lambda \log N \big\}\big| \leq C \exp\big( -c \textstyle\frac{\lambda }{\|f\|_\infty}\big)|Q|, \qquad \forall \lambda>0.
\end{equation}
\end{theorem}
\begin{remark} A simpler proof of this result for $T_{\V}=H_{\V}$ is given in  paragraph \ref{aoss}, exploiting the almost $L^p$-orthogonality principle in Theorem \ref{sect2-Lp-ortho}.
\end{remark}

\begin{definition}
\label{lacunary-def}
An ordered  set $\V=\{v_1,v_2,\ldots\} \subset S^1$  (finite or countably infinite)  is called \emph{lacunary} with node $v_\infty \in S^1$ if
\begin{equation}
\label{lacunary-defeq}
|v_{j+1}-v_\infty| \leq \frac{1}{2} |v_j-v_\infty|, \qquad j=1,2,\ldots.
\end{equation}
\end{definition}
\begin{definition}
\label{vargas-def} \cite{KATZ2}
A  finite set $\V \subset S^1$ with $\#\V=N$ is a  \emph{Vargas set} with constant $Q$ if
\begin{equation}
\label{vargas-defeq}
\max \{\#\V': \V'  \textrm{ is  a lacunary subset of } \V \} = Q \log N.
\end{equation}
The Vargas sets with constants $Q=O(1)$ independent of $N$ will simply be referred to as Vargas sets.
Examples include uniformly distributed sets of directions $
\V_N= \big\{v_j=  e^{2\pi i\frac{j}{N}}, j=0, \ldots, N-1\big\}
$
which are easily seen to be  Vargas sets with $Q=1$. The  $N$-truncations of Cantor sets are also Vargas sets with $Q=4$. See Subsection \ref{Vargas-ss} for more details.
\end{definition}
Finite lacunary sets and Vargas sets are at the opposite ends of the spectrum in terms of  maximum cardinality of lacunary subsets. In fact, every set of $N$ directions is ``at least $\frac13$ Vargas'', in the sense that it contains a lacunary subsets with $\frac{\log N}{3}$ elements (a proof is given in Subsection \ref{Vargas-ss}). In these two extreme (but rather relevant) cases, we are able to  improve the result of Theorem \ref{main-ub1} and obtain essentially sharp bounds.
\begin{theorem}  \label{main-ub2}
Let $\V$ be a lacunary set of $N $ directions. We have the sharp bound
\begin{equation} \label{main-ub2eq}
  \|T_\V \|_{p\to p} \lesssim_p \sqrt{\log N}, \qquad 1<p<\infty.
\end{equation}
\end{theorem}
%
\begin{theorem}  \label{main-ub3}
Let $\V$, $ \#\V=N$, be a Vargas set with constant $Q$.  We have the following  bound  \begin{equation} \label{main-ub3eq}
\|H_\V \|_{2\to 2,\infty}  \lesssim_Q \sqrt{\log N} (\log \log N)^6.
\end{equation}
In particular, if $V$ is uniform or a Cantor set then the following essentially sharp bound holds
$$\|H_\V \|_{2\to 2,\infty}  \lesssim \sqrt{\log N} (\log \log N)^6.$$
\end{theorem}
\begin{remark}
The (essential) sharpness of Theorem \ref{main-ub3} follows immediately via comparison with \eqref{karag}. It is not clear whether the $\log \log N$ term can be eliminated.

Note also that the bound in Theorem \ref{main-ub3} looses its strength as $Q$ gets large. 
\end{remark}
Interestingly, Theorem \ref{main-ub2} and Theorem \ref{main-ub3} show that uniformly distributed sets and lacunary sets of $N$ directions have the same quantitative behavior on weak $L^2$ (this is not the case for strong $L^2$).

The  methods of proof  of Theorem \ref{main-ub3} can also be easily modified to prove the estimate
$$
\|H_\V f \|_{p}  \lesssim_{p} (\log N)^{\frac{1}{p'}} (\log \log N)^{C_p}, \qquad 2<p<\infty,
$$
for each  Vargas set  of cardinality $N$. We strongly believe these methods are sharp, and thus that the exponent of $\log N$ is the correct one for each $p>2$.
We are  thus led to the following conjecture, which, as our results show, holds for lacunary sets and (at least in the case of the Hilbert transform) Vargas sets. \begin{cj}\label{general_conj} Let $T_\V$  be defined as in \eqref{sect1-mdm} and $\#\V=N$.   We have the bounds $$
\|T_\V \|_{2\to 2,\infty}  \lesssim \sqrt{\log N}(\log\log N)^C,\quad\;  \|T_\V \|_{p\to p} \lesssim_p(\log N)^{\frac{1}{p'}}(\log\log N)^{C_p},  \quad 2<p<\infty.
$$
\end{cj}
In light of our Theorem \ref{main-ub1}, the second conjectured inequality will follow if the first inequality is established.

\subsection{The methods and the plan of the paper} Theorem \ref{main-ub1} and \ref{main-ub2} are proved by classical harmonic analysis methods.
  The main tool behind the proof of Theorem \ref{main-ub1}, which is given in Section 2,  is an exponential-type good-lambda inequality for one-dimensional singular integrals by Hunt \cite{HUNT}, which allows to estimate the contribution of the $N$ singular integrals $T_v$ with a loss which is only logarithmic.

Theorem \ref{main-ub2} is proved in Section 3. Via an application of the exponential square integrability inequality by Chang, Wilson, and Wolff \cite{CWW}, we lose a $\sqrt{\log N}$ and reduce $T_\V$ to a square function  in the frequency Littlewood-Paley annuli, which is shown to be uniformly bounded in $N$ when $\V$ is lacunary.

On the other hand, Theorem \ref{main-ub3} is proved by means of time-frequency analysis in the phase plane.  We use  a dicretization for $H_\V$ very similar to the one introduced by Lacey and Li in their pioneering work \cite{LL1} and \cite{LL2}. The resulting model sum closely resembles the one introduced in \cite{LT} for Carleson's operator (see Sections 5 and 6). However, the \emph{tiles} appearing in the model sum for $H_{\V}$ are a bi-parameter family, and are allowed to be oriented along each direction in $\V$. These model sums are chopped into \emph{single tree} operators. Our analysis is somewhat similar to the one from \cite{DLTT} which contains a new proof of the boundedness of the Carleson operator. The approach from \cite{DLTT} is to show that for each point $x$, the contribution to the Carleson operator comes essentially from one tree. For Vargas sets, there will be at most $\log N$ trees contributing to $H_{\V}$. This upper bound is further lowered to $\sqrt{\log N}$ by interpolating with a Bessel inequality.

One of the things that distinguishes our approach  from the ones in \cite{BaTh}, \cite{CD10}, \cite{LL1} and \cite{LL2} is the fact that our trees will be of product nature (i.e. bi-parameter). The bulk of the argument is given in Section 7.
 \section{Proof of Theorem \ref{main-ub1}}
Throughout this section, we use the notations
$$
\mu_Q(A) = \int_{Q} \cic{1}_A(x) \; \frac{\d x}{|Q |}, \qquad
\|f\|_{L^p(Q)} := \left(  \int_{Q} |f(x)|^p \,\d\mu_Q(x)  \right)^{\frac1p}=\left(  \int_{Q} |f(x)|^p \,\frac{\d x}{|Q |}\right)^{\frac1p}$$
for $A \subset \R^2$ and $1\leq p< \infty.$
Before we enter the proof, we recall two results we   cited in the introduction, which will be used below:  the sharp bound
\begin{equation}
\label{HE2-L2}
\|T_\V f \|_{2\to 2} \lesssim \log N
\end{equation}
which has been proved in \cite{CD10}, and
\begin{equation}
\label{HE2-MF}
\|M_\V f\|_{p\to p} \leq C (\log N)^{\frac1p} \|f\|_p, \qquad 2<p<\infty
\end{equation}
which has first been established in \cite{KATZ1}.
\vskip2mm \noindent

\subsection{Maximal truncations of CZ kernels and a good-$\lambda$ inequality}
  To prove Theorem \ref{main-ub1}, we   need   two propositions.   The first is a consequence of the sharp weighted bound for maximal truncations of singular integrals in terms of the weight characteristic. The second  is a reformulation of a classical   good-lambda inequality by Hunt \cite{HUNT}.
\begin{proposition} \label{HE2-Tstar} We have the bound $
\|T_\star f\|_{p} \leq C p  \|f\|_p, $  for each $ 2 < p <\infty.
$
\end{proposition}
\begin{proposition} \label{HE2-Hunt}
There exist positive absolute constants $c_0,c_1$ such that  for all $f \in \C^\infty_0(\R)$, all $\lambda>0$ and all $\gamma>4$,
\begin{equation}  \label{huntmineeq}\big|\{x \in \R: |Tf(x)|>\gamma\lambda , M  f(x)<  c_0 \lambda\} \big|   \leq  C\e^{-c_1 \gamma} \big|\{x \in \R: |T_\star f(x)|> \lambda  \} \big|.
\end{equation}
\end{proposition}
 \begin{proof}[Proof of Proposition \ref{HE2-Tstar}] As in \cite{PF}, we follow Rubio de Francia's method. The recent article \cite{WEIGHT1} contains the sharp bound
\begin{equation}
\label{HE-2} \|T_\star f \|_{L^2(w)} \leq C [w]_{A_2}\|f\|_{L^2(w)}, \end{equation}
where $[w]_{A_2}$ denotes the $A_2$ characteristic of the weight $w$. Let $  q=(p/2)' >1$. Choose $v \in L^{q}(\R),$   such that $\|v\|_q=1$ and
$$
\|T_\star f\|_{p}^2 = \int_R |T_\star f(x)|^2 v(x) \d x.
$$
By interpolating the trivial $L^\infty$ bound with the endpoint $
\|Mf\|_{1} \leq C \|f\|_{L\log L}
$, one obtains the (sharp) bound
 $\|M \|_{q\to q} \leq C p$. Therefore, defining
$$
w = \sum_{k=0}^{\infty} \frac{1}{(2 \|M \|_{q\to q} )^k}\underbrace{M \circ \cdots \circ M}_{ k\, \mathrm{times}}  v
$$
we have
$$
v(x) \leq w(x) \;\; \mathrm{a.e.}, \quad \|w\|_q \leq 2,  \quad M  w(x) \leq  C p w(x) \;\; \mathrm{a.e.}, \quad [w]_{A_2} \leq Cp.
$$
We can conclude the proof with the chain of inequalities
\begin{align*}
\|T_\star f\|_{p}^2 = \int_R |T_\star f(x)|^2 v(x) \d x & \leq \int_R |T_\star f(x)|^2 w(x) \d x \\ & \leq C ^2 [w]_{A_2}^2 \int_R |  f(x)|^2 w(x) \d x \\ &\leq  C ^2 p^2\|f\|^2_{p},
\end{align*}
where we used \eqref{HE-2} in going from the first to the second line.
\end{proof}
\begin{proof}[Proof of Proposition \ref{HE2-Hunt}]
 Let us call $G$ the set on the right-hand side of \eqref{huntmineeq}. We can write
$G$ as the countable union of intervals $J$ with disjoint interiors such that $3J \not \subset G$. Therefore, \eqref{huntmineeq} will follow by summing up the estimate
\begin{equation}  \label{huntmineeqloc}\big|\{x \in J: |Tf(x)|>\gamma\lambda , M  f(x)<  c_0 \lambda\} \big|   \leq  C\e^{-c_1 \gamma}|J|.
\end{equation}
Let us prove \eqref{huntmineeqloc}. Choose a point $\bar x \in 3J \cap G^c$ and let $\bar J$ be the interval of length $6|J|$ centered at $\bar x$. Set
$$
f= f_1 + f_2, \qquad f_1 = f \cic{1}_{\bar J}, \, f_2 = f \cic{1}_{(\bar J)^c}.
$$
It is clear that
\begin{equation} \label{hm1} \eqref{huntmineeqloc}_{\mathsf{LHS}} \leq \big|\{x \in \bar J: |Tf_1(x)|>{\textstyle \frac{\gamma}{2}}\lambda , M  f(x)<  c_0 \lambda\} \big|   + \big|\{x \in  J: |Tf_2(x)|>{\textstyle \frac{\gamma}{2}} \lambda\}\big|
\end{equation}
Rescaling  Proposition 2 of \cite{HUNT}  on $\bar J$, we get that
\begin{equation} \label{hm2}  \big|\{x \in \bar J: |Tf_1(x)|>{\textstyle \frac{\gamma}{2}}\lambda , M  f(x)<  c_0 \lambda\} \big|   \leq A \e^{-\frac{a_0}{2c_0}\gamma} |\bar J| \leq 6A \e^{-\frac{a_0}{2c_0}\gamma} |J|.
\end{equation}
for some positive absolute constants $A,a_0$. If we prove that the set in the second term of $\eqref{hm1}_{\mathsf{RHS}}$ is empty for some choice of $c_0$, \eqref{huntmineeqloc} will follow from \eqref{hm2}, with $C=6A, c_1= \frac{a_0}{2c_0} $.
This is done as follows. It is known that (see \cite{STEIN1}, p.\ 208, eq.\ (35))
\begin{equation} \label{hm3}
|Tf_2(\bar x) - Tf_2(x)| \leq A_1 \inf_{J} M f \leq A_1 c_0 \lambda, \qquad \forall x \in J,
\end{equation}
for some absolute constant $A_1>0$. Here the second inequality follows because we can work under the assumption that the set in $\eqref{huntmineeqloc}_{\mathsf{LHS}}$ is nonempty. Also note that
$$
Tf_2(\bar x) = T_\eps f(\bar x), \qquad \eps = 6|J|
$$
and by our choice of $\bar x \notin G$, $|T_\eps f(\bar x)| \leq T_\star f(\bar x) \leq \lambda$. Hence, if we set $c_0=\frac{1}{A_1}, $ we get
$$
|Tf_2( x)| \leq |Tf_2(\bar x) - Tf_2(x)| + T_\star f(\bar x) \leq 2\lambda \leq \frac{\gamma}{2} \lambda , \qquad \forall x \in J,
$$
and this completes the proof.
\end{proof}

\subsection{Proof of Theorem \ref{main-ub1}}
By a limiting argument, we can assume that $f$ is supported on  a square $Q \subset \R^2$  of (finite) sidelength $\ell(Q)$. We first treat the part of  $T_\V f $ outside $2Q$.
Let  $x \in 2^{k+1}Q\backslash 2^kQ$, $k\geq 1$. We have, using the decay properties of the kernel $K$,
\begin{equation}
\label{HE2-PF8}
|T_vf(x)| \leq  \int_{\delta_1(x)}^{\delta_2(x)}  |f(x+tv)| |K(t)|\,\d t \leq \frac{C}{2^k \ell(Q)} \int_{-2^k\ell (Q)}^{2^k\ell (Q)} |f(x+tv)|  \,\d t \leq CM_v |f|(x).
\end{equation}
 Here we assumed that the line $\{x+tv:t \in \R\}$ intersects $Q$ at all $t$ between ${\delta_1(x)}$ and ${\delta_2(x)}$, with ${\delta_1(x),\delta_2(x)}\sim 2^k\ell(Q)$  (otherwise \eqref{HE2-PF8}  holds trivially).
Therefore
\begin{equation}
\label{HE2-PF9}
\|(T_\V f) \cic{1}_{\R^2\backslash 2Q}\|_{p} \leq C \|M_\V |f| \|_{p}\leq C(\log N)^{\frac1p}\|f\|_p.
\end{equation}
We now show that
\begin{equation} \label{HE2-expprop}
\|T_\V f\|_{L^p(2Q)} \leq C p  \log N  \|f\|_{L^p(Q)}, \qquad \forall 2<p<\infty;
\end{equation}
this  and \eqref{HE2-PF9} immediately imply  \eqref{main-ub1eq}. The weak-type endpoint result \eqref{main-ub1ep} is then obtained as an easy consequence of \eqref{HE2-expprop} (proceeding as in, for instance, \cite[Sect.\ IV.1.3]{STEIN1}).

In short, to prove \eqref{HE2-expprop}, we apply the exponential good-lambda inequality of Proposition \ref{HE2-Hunt} for each $T_v$, which acts separately on each fiber of $f$ along $v$, to excise $N$ different exceptional sets. Due to the exponential decay in the good-lambda inequality, the final loss is only logarithmic.
\begin{proof}[Proof of  \eqref{HE2-expprop}]  The range of exponents is open, so by Marcienkiewicz interpolation it suffices to prove the weak-type bound
\begin{equation} \label{HE2-PF0}
 \mu_{2Q}\big(\big\{x \in 2Q: T_\V f (x) > \lambda\log N  \big\}\big) \leq
\frac{C^p p^{p}}{\lambda^p} \|f\|_{L^p(Q)}^p, \qquad \lambda>0, \, 2<p<\infty.
\end{equation}
 For $\lambda>0$,  split ($c_0$ below is the absolute constant appearing in Proposition \ref{HE2-Hunt})
$$f= f_1+ f_2, \qquad f_1(x)= f(x)\cic{1}_{\{|f|\leq c_0\lambda\}}, \quad f_2= f-f_1,$$
We get at once that
\begin{align}& \label{HE2-PF1}
\sup_{x \in Q} |f_1(x)| \leq c_0 \lambda,
\\ &  \label{HE2-PF2}
\|f_1\|_{L^p(Q)} \leq \|f \|_{L^p(Q)},
\\  \label{HE2-PF3}
& \|f_2\|_{L^2(Q)}^2\leq C \lambda^{2-p} \|f \|_{L^p(Q)}^p, \\ \label{HE2-PF4}
& \mu_{2Q}\big(\big\{  T_\V f > 2\lambda\log N  \big\}\big) \leq  \mu_{2Q}\big(\big\{  T_\V f_1  >  \lambda\log N  \big\}\big) +  \mu_{2Q}\big(\big\{  T_\V f_2  >  \lambda\log N  \big\}\big)
\end{align}
We estimate the second term in the right-hand side of \eqref{HE2-PF4}; using the $ L^2$ bound of \eqref{HE2-L2}, and \eqref{HE2-PF3},
\begin{equation} \label{HE2-PF5}
\mu_{2Q}\big(\big\{  T_\V f_2  >  \lambda\log N  \big\}\big) \leq \frac{C}{\lambda^2}\|f_2\|_{L^2(Q)}^2 \leq   \frac{C}{\lambda^p}\|f\|_{L^p(Q)}^p.
\end{equation}
We turn to the first  term in the right-hand side of \eqref{HE2-PF4}. Preliminarily observe that
$$
\mu_{2Q}\big(\big\{  T_\V f_1  >  \lambda\log N  \big\} \big)\leq \mu_{2Q}\big(\big\{ T_\V f_1     > \lambda\log N  ,  M_\V f <c_0\lambda\big\}\big)  + \mu_{2Q}\big(\big\{   M_\V f >c_0\lambda\big\}\big) ,
$$
and \eqref{HE2-PF1} implies that the second set on the right hand side is empty. Thus, in view of \eqref{HE2-PF5},   \eqref{HE2-PF0} follows from the estimate
\begin{equation}
\label{HE2-exphunt}
 \mu_{2Q}\big(\big\{x \in \R^2:T_\V g (x)   > \lambda\log N  ,  M_\V g(x)<c_0\lambda\big\}\big) \leq\frac{ C^p p^p}{\lambda^p}\|g\|_{L^p(Q)}^p,
\end{equation}
applied to $g=f_1$. We have completed the proof   of \eqref{HE2-PF0} (and thus, of Theorem \ref{main-ub1}), up to showing \eqref{HE2-exphunt}, which we do at the end of the section.
\end{proof}
\begin{proof}[Proof of \eqref{HE2-exphunt}]
We claim that for each $v \in S^1$
\begin{equation}\label{HE2-Claim}
\mu_{2Q}\big(\{x \in \R^2: |T_v(g)(x)|>\gamma\lambda, M_v g(x)< c_0 \lambda\}\big)  \leq C\e^{-c_1 \gamma}\frac{ (Cp)^p\|g\|_{L^p(Q)}^p}{\lambda^p}.
\end{equation}
To prove the claim,  assume without loss of generality $v=(1,0)$. Denote by $g_{x_2}(\cdot)$ the function $
x_1 \mapsto  g(x_1,x_2)$. Then
\begin{align*}
&\quad \mu_{2Q}\big(\{x \in \R^2: |T_v(g)(x)|>\gamma\lambda, M_v g(x)< c_0 \lambda\} \big)\\ & \leq \frac{1}{4|Q|} \int_\R |\{x_1 \in \R: |T(g_{x_2})(x_1)|>\gamma\lambda  , M_v g_{x_2}(x_1)< c_0 \lambda\}| \, \d x_2
\\& \leq \frac{1}{4|Q|} \int_\R C\e^{-c_1 \gamma} |\{x_1 \in \R: |T_\star(g_{x_2})(x_1)|> \lambda \}|  \,\d x_2\\   & \leq \frac{C}{|Q|} \int_\R (Cp)^p\e^{-c_1 \gamma} \frac{\|g_{x_2}(\cdot)\|_p^p}{\lambda^p}  \,\d x_2\leq \frac{C}{|Q|}\e^{-c_1 \gamma}\frac{ (C p)^p\|g\|_p^p}{\lambda^p}=  C\e^{-c_1 \gamma}\frac{ (Cp)^p \|g\|_{L^p(Q)}^p.}{\lambda^p}
\end{align*}
where the third to last inequality is an application of Proposition \ref{HE2-Hunt}, and last line follows from  the bound of Proposition \ref{HE2-Tstar} for $T_\star$ on $L^p(x_1)$.  Now, using \eqref{HE2-Claim} for each $T_v$, we have
\begin{align}& \quad   \mu_{2Q}\big(\big\{x \in \R^2:T_\V g (x) \nonumber >(c_1)^{-1}\lambda\log N  ,  M_\V f(x)<c_0\lambda\big\}\big) \\& \leq
 \sum_{v\in \V} \mu_{2Q}\big(\{x \in \R^2: |T_v(g)(x)|>(c_1)^{-1}\lambda\log N  , M_v g(x)< c_0 \lambda\}  \big)  \nonumber
\\ & \leq CN \e^{- \log N}\frac{ (Cp)^p\|g\|_{L^p(Q)}^p }{\lambda^p}\leq  \frac{ (C p)^p}{\lambda^p}\|g\|_{L^p(Q)}^p, \nonumber
\end{align}
which completes the proof of \eqref{HE2-exphunt}.
\end{proof}

\section{Lacunary $\V$: Proof of Theorem \ref{main-ub2}}
\label{lac}
 We need to prove that for a set $\V$ of $N$ directions which is lacunary (see Definition \ref{lacunary-def})
$$
\|T_\V f\|_p \leq C_p \sqrt{\log N} \|f\|_p, \qquad 1<p<\infty
$$
In the related paper \cite{CD10}, the $L^2$-estimate of the maximal multiplier \eqref{sect1-mdm} is reduced (with a $\sqrt{\log N}$ loss)  to the $L^2$ bound for the corresponding square function, via an application of the Chang-Wilson-Wolff inequality \cite{CWW}. See also \cite{GRAFAKOS}, where the same method had been first applied to the study of maximal one-dimensional multipliers, and \cite{CD11}. We apply the same reduction, and then show that the resulting square function   is bounded on all $L^p, $ $1<p<\infty$ uniformly in the cardinality of $\V$ (see \eqref{Lacunary:SFE}).
\subsection{Reduction to a square function estimate}\label{Lacunary:ss1}
Let $\Phi$ be a Schwartz function supported on $[1/2,2]$ and such that
$$
\sum_{k \in \mathbb Z} \Phi(2^{-k}\xi )=1, \qquad \forall \xi \in \R\backslash\{0\}.
$$
For a function $f\in C^\infty_0(\R^2)$, set $\widehat{S_k f}(\xi)=\hat f(\xi ) \Phi(2^{-k}|\xi |),$
and note that
\begin{equation}
\label{Lacunary:LP}
\|f\|_p \sim \bigg\|\Big(\sum_{k \in \mathbb{Z}} |S_k f|^2\Big)^{\frac12} \bigg\|_p, \qquad 1 <p<\infty.
\end{equation}
 We summarize the reduction to a square function in the following proposition; the proof is contained in the reference \cite{GRAFAKOS}, see also \cite{CD10}.
\begin{proposition}Let $P_i$, $i=1,\ldots,N$ be linear (uniformly) bounded multiplier  operators on $\R^2$. \label{Lacunary:CWW} Denote by $\mathcal{M}_\beta=M_{\beta}\circ M_{\beta}\circ M_{\beta}  $, where $M_\beta f :=( M [f^\beta])^{\frac1\beta}$.
Then, for  each $1<p<\infty$ and each $\eps >0$ we have
\begin{equation}
\Big\|\sup_{i=1,\ldots,N} |P_i f|\Big\|_{p} \lesssim_{\epsilon,p} (\log N)^{\frac12} \bigg\|  \bigg( \sum_{k \in \mathbb Z} \Big(\mathcal{M}_{1+\eps } \Big[ \sup_{i=1,\ldots,N} |P_i(S_k f)|\Big]\Big)^2 \bigg)^{\frac{1}{2}} \bigg\|_p.
\end{equation}
\end{proposition} \noindent
By means of Proposition \ref{Lacunary:CWW}, and a subsequent use of the Fefferman-Stein inequality (with $1+\epsilon<p$) we get
$$
\bigg\|\Big(\sum_{k \in \mathbb{Z}} \mathcal{M}_{1+\eps}[ g_k]^2\Big)^{\frac12} \bigg\|_p \lesssim_p \bigg\|\Big(\sum_{k \in \mathbb{Z}} |g_k|^2\Big)^{\frac12} \bigg\|_p,
$$
followed by \eqref{Lacunary:CWW}, we reduce Theorem \ref{main-ub2}  to the proof of the estimate
\begin{equation} \label{Lacunary:SFE}
 \bigg\|  \Big( \sum_{k \in \mathbb Z} |T_\V(S_k f)|^2 \Big)^{\frac{1}{2}} \bigg\|_p \lesssim_p  \|f\|_{p}, \qquad 1<p<\infty.
\end{equation}
Let us recall once again that the easy case $p=2$ of \eqref{Lacunary:SFE} was proved in \cite{CD10}.

\subsection{Proof of the square function estimate \eqref{Lacunary:SFE} } We can harmlessly assume  that the multiplier $m$ is supported on $[0,\infty)$ and that $\hat f$ is supported in the first quadrant.
We  take the node of  $\V$  to be $v_\infty=(1,0)$   and order $\V=\{v_{j}=\e^{2\pi i \theta(v_j)}\}$  counterclockwise.

Since $\theta(v_\infty)=0$,  \eqref{lacunary-defeq} reads $|\theta(v_{j+1})| \leq 2^{-1}|\theta(v_j)|$, $j \geq 1$. By eventually splitting $\V$ into the three subcollections $ \{v_{j} \in \V: j \mod 3 = i \}$, $i\in \{0,1,2\}$,  we can work under the stronger separation property
\begin{equation}
\label{Lacunary:lcn1}
|\theta(v_{j+1})| \leq 2^{- 3}|\theta(v_j)|,
\qquad j \geq 1.
\end{equation}
We denote $\delta=|\theta(v_1)|$. It is clear that for each interval $I_\ell:=(2^{-3(\ell+1)}\delta,2^{-3\ell}\delta,]$, $\ell \geq 1$, there is at most one $j(\ell)$ such that $\theta(v_{j(\ell)} )\in I_\ell$. Whenever $I_\ell \cap \{\theta(v): v \in  \V\}=\emptyset$ for (at least two consecutive)  $\ell,\ell+1, \ldots, \ell+m$, we add to $\V$ the $m$ directions  with angles $\theta \in   \{ -  2^{-3\ell}, \ldots, - 2^{-3(\ell+m-1)} \}$. We continue to call $\V$ the resulting collection and re-index  the $v_j$ so that they are ordered counterclockwise. By doing so, we have obtained that   \begin{equation}
\label{Lacunary:lcn}
2^{-3(j+1)}\delta \leq \delta_j:=|\theta(v_{j+1})- \theta(v_{j})| \leq  2^{-3j}\delta, \qquad j \geq 1.
\end{equation} It is handy to define $w_j= \theta(v_j)+ \frac{1}{4}$, since the assumption on the support of the multiplier $m$   implies that for $\xi=|\xi|\e^{2\pi i \theta(\xi)}$ in the first quadrant
 \begin{equation}
\label{Lacunary:supp} m(\xi \cdot v_j) \neq 0 \implies \theta(\xi) \in (0,w_j).
\end{equation}
Accordingly, define the intervals
$$
\alpha_0=(0,w_1], \quad \alpha_j=(w_j,w_{j+1}], \;j=1, \ldots, \#\V-1,\quad  \alpha_{\#\V}=(w_{\#\V}, w_\infty].
$$
For an interval $I=(\alpha_\ell, \alpha_r) \subset [0,1]$, define the   frequency cutoff to the cone  $\{e^{2\pi i\theta}:\theta\in I\}$
$$\widehat{G_I f}(\xi)= \hat f(\xi) \cic{1}_I (\theta(\xi)), \qquad \xi=|\xi|\e^{2\pi i \theta (\xi)}.$$  We simply write $G_j$ for $G_{\alpha_j}$.
%
The bulk of the proof is contained in the next two results.
In the first lemma, which is proven in Subsection \ref{lacchap-2},  a pointwise bound on $T_\V(S_k f)$ is given,  in terms (essentially) of the maximal function along the direction $v_\infty$.
\begin{lemma}\label{Lacunary-pointwise}
Define
\begin{equation} \label{Lacunary-trunc}
 f_{\mathsf{ev}} = \sum_{j} G_{2j } f,\qquad  f_{\mathsf{odd}} = \sum_{j} G_{2j+1} f,
\end{equation}
and adopt the notation (here $v^\perp$
is the direction orthogonal to $v$)
$$
M^{\mathsf{bi}}_{v} f(x)= \sup_{t,s>0}\frac{1}{4ts} \int_{-t}^{t}\int_{-s}^s |f(x+\tau v+\sigma v^\perp)| \, \d \tau \d \sigma, \qquad x \in \R^2,
$$
for the bi-parameter maximal function in the coordinates of $v$. Then
\begin{equation} \label{Lacunary-pointwiseEQ}
T_\V(S_k f)(x) \lesssim    M^{\mathsf{bi}}_{v_\infty}  \big[  S_k f_{\mathsf{ev}} \big](x) +   M^{\mathsf{bi}}_{v_\infty}  \big[   S_k f_{\mathsf{odd}} \big](x) , \qquad x \in \R^2, \,k\in \Z.
\end{equation}
\end{lemma} The second lemma deals with the boundedness of the rough cutoff multiplier $\sum_j \eps_j G_j$. The proof is given in Subsection \ref {lacchap-3}.
\begin{lemma}\label{Lacunary-conemult}
We have the estimate
\begin{equation} \label{Lacunary-conemult1}
 \sup_{\eps_j=0,\pm1} \Big\| \sum_{j} \eps_j G_j f\Big\|_p \lesssim  \|f\|_p, \qquad 1<p<\infty.
\end{equation}
\end{lemma}
\begin{proof}[Conclusion of the proof of \eqref{Lacunary:SFE}]
Lemma \ref{Lacunary-conemult} immediately implies the bounds
\begin{equation} \label{Lacunary:SFE1}
\|f_{\mathsf{ev}}\|_p, \; \|f_{\mathsf{odd}}\|_p \lesssim \|f\|_p, \qquad 1< p <\infty.
\end{equation}
We conclude, using the Fefferman-Stein inequality on the vector valued $M^{\mathsf{bi}}_{v_\infty}$, that
\begin{align*}
&\quad\bigg\|  \Big( \sum_{k \in \mathbb Z} |T_\V(S_k f)|^2 \Big)^{\frac{1}{2}} \bigg\|_p \\ & \lesssim \bigg\|  \Big( \sum_{k \in \mathbb Z} \Big[ M^{\mathsf{bi}}_{v_\infty} \big[  S_k f_{\mathsf{ev}} \big] \Big]^2 \Big)^{\frac{1}{2}} \bigg\|_p + \bigg\|  \Big( \sum_{k \in \mathbb Z} \Big[  M^{\mathsf{bi}}_{v_\infty}  \big[  S_k f_{\mathsf{odd}} \big] \Big]^2 \Big)^{\frac{1}{2}} \bigg\|_p
\\ & \lesssim \bigg\|  \Big( \sum_{k \in \mathbb Z} \big| S_k f_{\mathsf{ev}} \big| ^2 \Big)^{\frac{1}{2}} \bigg\|_p +  \bigg\| \Big( \sum_{k \in \mathbb Z} \big| S_k f_{\mathsf{odd}} \big| ^2 \Big)^{\frac{1}{2}} \bigg\|_p \\ & \lesssim \|f_{\mathsf{ev}}\|_p+\|f_{\mathsf{odd}}\|_p \lesssim \|f\|_p.
\end{align*}
 This completes the proof of \eqref{Lacunary:SFE}, and in turn, of Theorem \ref{main-ub2}.
\end{proof}
\subsection{Proof of Lemma \ref{Lacunary-pointwise}} \label{lacchap-2}
The estimate \eqref{Lacunary-pointwiseEQ} is an easy consequence of
\begin{equation}
\label{Lacunary:PW1}
\big|T_{v_j} [S_k f_{\mathsf{ev}}] (x)\big| \lesssim      M^{\mathsf{bi}}_{v_\infty}  \big[  S_k f_{\mathsf{ev}} \big](x), \qquad \forall x \in \R^2,\, v_j \in \V,
\end{equation}
and of the analogue estimate for $S_k f_{\mathsf{odd}}$.
Observe that $T_{v_{2j}} [S_k f_{\mathsf{ev}}]= T_{v_{2j-1}} [S_k f_{\mathsf{ev}}]$, since $\widehat{ f_{\mathsf{ev}}}$ is zero on the cone bordered by $w_{2j-1} $ and $w_{2j}$. Thus it suffices to consider odd $j$.
\noindent
We now show \eqref{Lacunary:PW1}. Let $j$ be fixed throughout.  We will omit some dependence on $j$ in the notation. Denote by $$ (\eta_{1 }, \eta_{2 }):=\{tw_{j}:t>0\} \cap \{|\xi|=2^k\}, \qquad (\zeta_{1 }, \zeta_{2 }):=\{t w_{j+1}   :t>0\} \cap \{|\xi|=2^{k+1}\},
$$
As a consequence of \eqref{Lacunary:lcn1}, we have that
$$
0 <\zeta_1 = 2^{k+1} \cos(w_{j+1}) \leq  2^{k+1}  |\theta(v_{j+1})|  <2^{k-2}|\theta(v_{j})|  \leq  2^k \cos(w_j)= \eta_1 \sim 2^k \delta 8^{-j}
$$
Now, let $\psi_j$ be a positive Schwartz function on $\R$ which is equal to $1$ on $[2^k 8^{- j -1}\delta,2^k ]$ and vanishes outside of  $[2^k 8^{-j -2}\delta ,2^{k+1} ]$. The idea is that we can insert the vertical frequency cutoff given by $\psi_j(\xi_1)$ harmlessly, because (see \eqref{Lacunary:supp}) $$
\widehat{ [S_kf_{\mathsf{ev}} ]} (\xi) m(\xi \cdot v_j) \neq 0 \implies \xi_1 \geq \eta_1 \implies  \psi_j(\xi_1)=1,
$$
and therefore $\widehat{T_{v_j} [S_kf_{\mathsf{ev}} ]}(\xi) = \widehat{S_kf_{\mathsf{ev}}}(\xi) \mu_j(\xi)$, where we set
$
\mu_j(\xi) = \Phi(2^{-k}|\xi|)\psi_j(\xi_1) m(\xi \cdot v_j).
$
The multiplier $\mu_j$ is easily seen to satisfy
$$
\big|\partial^\alpha_{\xi_1}\partial^\beta_{\xi_2} \mu_j(\xi_1,\xi_2)\big| \lesssim 8^{j\alpha}\delta^{-\alpha}2^{-k \alpha } 2^{-k \beta } , \qquad 0\leq \alpha,\beta \leq \kappa,
$$ for some constant $\kappa>3$.
Thus $K_j=\check\mu_j$ obeys the bound $$
|K_j (x_1,x_2)| \lesssim  2^{2k} \delta 8^{-j} \big| 1+2^k\delta 8^{-j} |x_1|+2^k|x_2| \big|^{-3}
$$
This implies
$$
|T_{v_j} [S_kf_{\mathsf{ev}} ]| (x_1,x_2)\lesssim| S_kf_{\mathsf{ev}} *K_j|(x_1,x_2)\lesssim M^{\mathsf{bi}}_{v_\infty}[S_kf_{\mathsf{ev}} ](x_1,x_2),
$$
as claimed in \eqref{Lacunary:PW1}. The proof of Lemma \ref{Lacunary-pointwise} is complete.
\subsection{Proof of Lemma \ref{Lacunary-conemult}} \label{lacchap-3} We need a smooth analogue of $G_I$. Let $\beta$ be a positive Schwartz function  with support in $ [-\frac12,\frac32] $ and such that
$
\sum_{\ell \in \Z} \beta( t - \ell) = \cic{1}_\R(t)
$; define the multiplier operator$$ \widehat{ G_I^\mathsf{s } f}(\xi):= \hat f(\xi)  \beta\Big(   \frac{\theta(\xi)-\alpha_\ell}{|I|}\Big), \qquad \xi=|\xi|\e^{2\pi i \theta(\xi)} \in \R^2. $$
Let us cover each $\alpha_j$ by 128 intervals of equal length $\alpha_{j,\ell} = (w_j + (\ell-1)\frac{\delta_j}{128}, w_j + \ell\frac{\delta_j}{128}\big]$, and denote
$  G_{\alpha_{j,\ell}}^\mathsf{s }$ simply by $G_{j,\ell} ^\mathsf{s }$.
By the definition of $\beta$, we have the identity \begin{equation}
\label{identity}
G_j f = G_j (G^{\mathsf{s}}_j f), \qquad  G_j^{\mathsf s} f:= \sum_{\ell=1}^{128} G_{j,\ell}^{\mathsf s} f .
\end{equation}
 The frequency support of $G_{j,\ell}^\mathsf{s }$ is contained in $2\alpha_{j,\ell}$ (the interval with same center as $\alpha_{j,\ell}$ and twice the length). By \eqref{Lacunary:lcn}, $\delta_j$ and $\delta_{j+1}$ are  within a factor of $16$, thus
$$
\sum_{j} \sum_{\ell=1}^{128} \cic{1}_{2\alpha_{j,\ell}}(t) \leq 32.
$$ The above observations and the smoothness of $\beta$  imply that  $G= \sum_{j,\ell} \eps_{j,\ell}G_{j,\ell}^{\mathsf{s}}$ is a (two-dimensional) H\"ormander-Mihlin multiplier, and therefore
\begin{equation}
 \sup_{\eps_j=0,\pm1}  \Big\|\sum_{j,\ell} \eps_{j,\ell}G_{j,\ell}^{\mathsf{s }}h\Big\|_p \lesssim_p \|h\|_p, \qquad 1<p<\infty;
\end{equation}
see \cite[Theorem IV.3]{STEIN2}. By randomization, we get the square function inequality
\begin{equation} \label{Lacunary:smooth}
 \bigg\|  \Big( \sum_{j,\ell} |G^\mathsf{s}_{j,\ell}(h)|^2 \Big)^{\frac{1}{2}}\bigg\|_p \lesssim \|h\|_p, \qquad 1<p<\infty.
\end{equation}
 We will need the vector-valued inequality of the Proposition below, which is a variant of the results of \cite{CF} and its refinement  \cite[Theorem VI.8.1]{GCRDF}. The full range $1<p<\infty$ is obtained by combining the method used in \cite{GCRDF} with the sharp weighted bound for the Hilbert transform in terms of the $A_2$ characteristic of the weight  \cite{PET}.
\begin{proposition} \label{Lacunary:VVpr}
For each $1< p <\infty $, we have that
\begin{equation} \bigg\|  \Big( \sum_{j} |G_j[h_j]|^2 \Big)^{\frac{1}{2}}\bigg\|_p \lesssim\bigg\|  \Big( \sum_{j} | h_j |^2 \Big)^{\frac{1}{2}}\bigg\|_p. \label{Lacunary:VV}
\end{equation}
\end{proposition}\noindent
Recall the identity \eqref{identity}. Applying \eqref{Lacunary:VV} for $h_j=G^{\mathsf{s}}_j g$, and subsequently using \eqref{Lacunary:smooth}, we get that for each $g\in L^2\cap L^p$
\begin{align*}\bigg\|  \Big( \sum_{j} |G_j g |^2 \Big)^{\frac{1}{2}}\bigg\|_p & = \bigg\|  \Big( \sum_{j} |G_j[G^{\mathsf{s}}_j g]|^2 \Big)^{\frac{1}{2}}\bigg\|_p \lesssim  \bigg\|  \Big( \sum_{j}  |G^{\mathsf{s}}_j g|^2 \Big)^{\frac{1}{2}}\bigg\|_p \\ & \nonumber
\lesssim\bigg\|  \Big( \sum_{j,k} | G^{\mathsf{s}}_{j,k} g |^2 \Big)^{\frac{1}{2}}\bigg\|_p \lesssim \|g\|_p
\end{align*}
\noindent
When $p=2$, the first and the  last quantities in the above display coincide, so that polarization  actually gives the equivalence
\begin{equation} c_p\|g\|_p\leq \bigg\|  \Big( \sum_{j} |G_j g |^2 \Big)^{\frac{1}{2}}\bigg\|_p\leq C_p \|g\|_p, \qquad 1<p<\infty. \label{Lacunary:VV3}
\end{equation}
Applying the left inequality in \eqref{Lacunary:VV3} to $g=\sum_j \eps_j G_j f$ then yields
$$
 \|g\|_p\lesssim \bigg\|  \Big( \sum_{j} |G_j g |^2 \Big)^{\frac{1}{2}}\bigg\|_p\leq \bigg\|  \Big( \sum_{j} |G_j f|^2 \Big)^{\frac{1}{2}}\bigg\|_p\lesssim \|f \|_p, \qquad 1<p<\infty.
$$
where the last step follows from the right inequality in \eqref{Lacunary:VV3}. The last display is exactly the assertion we had to prove.

\section{Grids, adaptedness, and a product John-Nirenberg inequality}
\label{Grids}
In this section, we define  product size and establish a version of the (product) John-Nirenberg inequality, Proposition \ref{JNpp}. Most of these results are well known from previous literature. Let us start with some definitions.
\subsection{Grids} A  one-dimensional \emph{grid} $\G$ is a collection of intervals with the property that
\begin{equation}
\label{grid-1} I, I' \in \G \implies I\cap I' \in \{\emptyset, I, I'\}.
\end{equation} For $i \in \{0,1,2\}$,  $\D^i=\big\{2^j\big(\ell+\frac{i(-1)^j}{3}+[0,1): j, \ell \in \mathbb Z\big)\big\}$ are three (dyadic) grids with the property that
for each interval $J \subset \R$ there exists some $I \in \D\cup\D^{1}\cup \D^{2}$ with $J \subset I \subset 9J$\footnote{By $CJ$ we mean the interval  with same center as $J$ and  $C$ times longer. If $R=I\times J$ is a rectangle, $CR=CI\times CJ$. }. For simplicity, we write $\D$ for $\D^0$; for a scale $\scl \in 2^{\mathbb Z},$ $\D^i_\scl=\{I \in \D^i: |I|=\scl\}$.

  A grid $\G$ is said to be $K$-sparse if for every two $I,J \in \G$
\begin{equation}
\label{grid-2}  |I| < |I'| \implies |I|\leq 2^K |I'|, \qquad |I|=|I'|\implies \dist(I,I')\geq 2^{K} |I'|.
\end{equation}
Right from the definition, if $\G$ is $(K+10)$-sparse and $I,I' \in \G,$ with $|I|=|I'|$, the dilates $2^{K+4}I$ and $2^{K+4 }I'$ do not intersect. On the other hand, $I,I' \in \G, $ $I \subset I'  \implies  2^{K+4} I \subset 2I'$. Using this property, it follows that if $\G$ is $(K+10)$-sparse,
\begin{equation}
\label{grid-3}
\forall I \in \G \; \exists J(I) \in \D^0\cup\D^{1}\cup \D^{2} \textrm{ with } 2^KI \subset J(I) \subset 9\cdot 2^KI,
\end{equation}{ and } $|I|=|I'| \implies J(I) \cap J(I')=\emptyset $, $I \subset I' \implies J(I) \subset J(I').$
Now, observe that every dyadic grid $\G$ can be written as the disjoint union of $K2^K$ $K$-sparse dyadic grids. It follows that $2^K\G:=\{2^K I: I \in \G\}$ can be written as the union of $3\cdot K  2^{K}$ subsets, each of which fits into a dyadic grid as described in \eqref{grid-3}.

\subsection{$L^2$-adaptedness}
Let $R=I_R \times J_R$ be a rectangle in $\R^2$ with orientation $(e_R,e_R^\perp)$. We denote by $c(R)$ its center and by $R^{(k)}$ the rectangle $2^kR$. Let $K$  be a large positive constant. We say that a Schwartz function $\psi_R$ is   $L^2$-\emph{adapted} to  the rectangle $R$ if for all $\alpha, \beta>0$ there exists  a constant $C_{\alpha,\beta}$ such that, in the coordinates ($e_R,e_R^\perp$),
\begin{equation}
\label{L2-ad}
|\partial_{x_1}^\alpha \partial_{x_2}^\beta\psi_R(x_1,x_2)| \leq  \frac{C_{\alpha,\beta}}{|I_R|^{\alpha+\frac12} |J_R|^{\beta+\frac12}} \Big(1+\frac{|x_1-c(R)_1|}{|I_R|} + \frac{|x_2-c(R)_2|}{|J_R|}  \Big)^{-K}.
\end{equation}
  Let $\RR$ be a collection of rectangles (with possibly different orientations). We say that $\{\psi_R\}_{R \in \RR}$, are  $L^2$-adapted to $\RR$ if each $\psi_R$ is $L^2$-adapted with same the constants $C_{\alpha,\beta},K$.

If $\{\psi_R\}_{R \in \RR}$ are $L^2$-adapted and for each $R$ we have $\int_{\R^2} \psi_R=0$, we say that $\{\psi_R\}_{R \in \RR}$ are \emph{wave packets} adapted to $\RR$. If the stronger condition
\begin{equation} \label{meanzero}   \int_{\{x+te_{R} : t \in \R\}} \psi_{R}  =  \int_{\{x+te_{R}^\perp: t \in \R\}} \psi_{R}    = 0, \qquad  \forall \, x \in \R^2
\end{equation}
holds, we say that $\{\psi_R\}_{R \in \RR}$ are product wave packets adapted to $\RR$.
\subsection{A product John-Nirenberg inequality}
Let $\RR$ be a collection of rectangles $R=I_R \times J_R$ with fixed orientation. Assume that $\I_\RR:=\{I_R:R\in\RR\}$ and $\mathcal{J}_\RR:=\{J_R:R\in \RR\} $ are each (respectively) subsets of a dyadic grid $\I$ (resp.\ $\mathcal J$). We will use the notation $$\sh(
\RR)= \bigcup_{R \in \RR} R $$
for the \emph{shadow} of such a  collection.
Given a collection of complex coefficients $B=\{b_R\}_{R \in \RR}$ and functions   $\{\psi_R\}_{R \in \RR}$, define
\begin{equation}  \label{Bdef}
\mathsf{B}_{\RR'}(x)= \sum_{R\in \RR'} b_R \psi_R(x), \qquad \RR'\subset \RR.
\end{equation}
Let us introduce
\begin{equation}  \label{sizedef}
\size(B) = \sup_{\RR'\subset \RR} \bigg( \frac{1}{|\sh(\RR')|} \sum_{R \in \RR'} |b_R|^2 \bigg)^{\frac12}
\end{equation}
In the following proposition, we show how $\size$ is akin to the product BMO norm.
\begin{proposition} Let $\{\psi_R\}_{R \in \RR}$ be product wave packets adapted to $\RR$, and with
\begin{equation}
\label{supp}
\mathrm{supp}\, \psi_R \subset R.
\end{equation}
We have the estimates \label{JNpp}
\begin{align}
\label{JNp}
 &\| \mathsf{B}_{\RR}\|_{p} \leq C p^2  \size(B) |\sh(\RR)|^{\frac1p}, \qquad 2 \leq p < \infty,
\\
\label{JNexp} & \big|\big\{x \in \sh (\RR): |B_{\RR}(x)| > \lambda \big\} \big| \leq C\exp\Big(-c \textstyle \sqrt{\frac{\lambda}{\size(B)}} \Big)|\sh (\RR)|.
\end{align}
\end{proposition}
Subsections \ref{ss44} and  \ref{ss45}  are devoted to the proofs (respectively)  of Proposition \ref{JNpp} and of the related product Chang-Wilson Wolff inequality (Proposition \ref{f-Delta}).
We can actually show that estimate \eqref{JNp} continues  to hold even when the restriction \eqref{supp} on the support of the $\psi_R$ is removed.  This will be done in Subsection \ref{ss46}.
\begin{remark}
We also detail a one parameter version of Proposition \ref{JNpp}, whose proof follows the same lines. Let    $\RR$ be a collection of rectangles as above, with   fixed eccentricity $\frac{|J_R|}{|I_R|}\equiv \ecc$. Let $\{\psi_R\}_{R \in \RR}$ be   wave packets adapted to   $\RR$, and with $
\mathrm{supp}\, \psi_R \subset R.
$
Then\begin{equation}
\label{JNexp1P}
\big|\big\{x \in \sh (\RR): |B_{\RR}(x)| > \lambda \big\} \big| \leq C\exp\Big(-c \textstyle  {\frac{\lambda}{\size(B)}} \Big)|\sh (\RR)|.
\end{equation}
\end{remark}
\subsection{Proof of Proposition \ref
{JNpp}} \label{ss44}
For simplicity, say $\I=\mathcal{J}=\D$. Proposition \ref{JNpp} relies upon two results.
 First, we establish an inequality similar to \eqref{JNp} for the related square function
\begin{equation}  \label{SBdef}
\mathsf{SB}_{\RR'}(x)= \bigg( \sum_{R\in \RR'}  \frac{|b_R|^2}{|R|} \cic{1}_R (x)\bigg)^{\frac12}, \qquad \RR'\subset \RR.
\end{equation}
\begin{lemma} \label{JNSBL}
We have the estimate \begin{equation} \label{JNSB}
\|\mathsf{SB_{\RR}}\|_{p} \leq C  p\,\size(B)  |\sh(\RR)|^{\frac1p} , \qquad 2< p <\infty.
\end{equation}
\end{lemma} \noindent
Second, we  obtain  a sharp (with respect to $p$) norm comparison between $\mathsf{B}_\RR$ and $\mathsf{SB}_\RR$.
 \begin{lemma} \label{B-SBL}
We have the estimate \begin{equation} \label{B-SB}
\|\mathsf{B}_{\RR}\|_{p} \leq C  p \|\mathsf{SB}_\RR\|_p, \qquad 2\leq p <\infty.
 \end{equation}
\end{lemma}
Thus, \eqref{JNp} is obtained by applying  Lemmata \ref{JNSBL} and \ref{B-SBL}. Then,   \eqref{JNexp} follows immediately from \eqref{JNp} by extrapolation.

The proofs of Lemma \ref{JNSBL} and
  Lemma \ref{B-SBL} are given at the end of this section.  Lemma \ref{B-SBL} will be obtained as a consequence
  of the product Chang-Wilson-Wolff inequality of \cite{PIPHER}, which we detail in Proposition \ref{f-Delta} in the    form of sharp (with respect to $p$, as $p \to \infty$) $L^p$ bounds between $f$ and its product martingale square function, following the approach of \cite{PF}.   We recall some notation.
For $I \in \D$, $h_I$ denotes the $L^2$ normalized Haar wavelet on $I$. For a rectangle $Q=I\times J$, $h_Q$ stands for the tensor product wavelet $h_I \otimes h_J$. We introduce the one and two-parameter martingale square functions
\begin{equation} \label{MSF}
(\Delta_1 f) (x_1) =  \bigg(\sum_{I \in \D} |\l f, h_I \r|^2 \frac{\cic{1}_I (x_1)}{|I|}  \bigg)^{\frac12}, \quad(\Delta_{12}) f(x)\bigg(\sum_{Q \in \D\times \D} |\l f, h_Q \r|^2 \frac{\cic{1}_Q (x)}{|Q|}\bigg)^{\frac12}.\end{equation}
\begin{proposition} \label{f-Delta}
We have the sharp inequality
\begin{equation} \label{f-Delta1}
\|f\|_p \leq C p\|\Delta_{12} f\|_p, \qquad 2\leq p < \infty.\end{equation}
\end{proposition} \noindent
The proof of Proposition \ref{f-Delta} is outlined in Subsection \ref{ss45}.

\begin{proof}[Proof of Lemma \ref{JNSBL}]  Let $p>2$ and $q=\frac p2$. It then suffices to show that
 \begin{equation}
\|(\mathsf{SB}_\RR)^2\|_{L^q(\sh(\RR))} \leq C  q\|\M_\RR\|_{q'\to q',\infty}  \size(B)^2 |\sh(\RR)|^\frac1q, \qquad 1\leq q <\infty.
\label{JNSB-1}
\end{equation}
Indeed, the bound $\|\M_\RR\|_{q'\to q',\infty} \leq Cq$ follows immediately via weak-type interpolation of the trivial $L^\infty$ bound with the endpoint for the strong maximal function on $\R^2$ [JMZ]
 \begin{equation}
\big|\big\{ x \in \sh (\RR): \M_\RR f(x) > \lambda \big\}\big| \leq C\int_{\sh (\RR)} \frac{|f(x)|}{\lambda} \log \Big( \e + \frac{|f(x)|}{\lambda}\Big)\, \d x.
\label{JNSB-2}
\end{equation} For simplicity, denote $K_q:=\|\M_\RR\|_{q'\to q',\infty}$ . To prove \eqref{JNSB-1}, we assume the following claim. \vskip1.5mm \noindent \textsc{Claim.} For each $\RR' \subset \RR$ there exists $\RR''\subset \RR'$ with
\begin{equation} \label{JN-induct}
  |\sh(\RR'')| \leq \frac{ |\sh(\RR')|}{2}, \qquad
\|(\mathsf{SB}_{\RR'})^2\|_q \leq  2^{\frac{1}{q'}}K_q  \size(B)^2|\sh(\RR')|^{\frac{1}{q}} + \|(\mathsf{SB}_{\RR''})^2\|_q.
\end{equation}
With the claim in hand, we can easily complete the proof. Apply the claim to $\RR'=\RR$; this yields an $
\RR'':= \RR_1$ with $\sh(\RR_1)\leq 2^{-1}\sh(\RR) $. Apply the claim with $\RR'= \RR_1$. This yields an $\RR''=\RR_2$ with  with $\sh(\RR_1)\leq 2^{-2}\sh(\RR) $. Iterating, one finally gets
\begin{equation} \label{JN-induct2}
\|(\mathsf{SB}_{\RR'})^2\|_q \leq  CK_q   \size(B)^2|  \sum_{k=0}^\infty 2^{-\frac k q}|\sh(\RR)|^{\frac{1}{q}} \leq \frac{CK_q}{1-2^{-\frac1q}}\size(B)^2|\sh (\RR)|^{\frac1q}
\end{equation}
and the proposition follows from the fact that $1-2^{-\frac1q}  \geq cq^{-1}$.

\noindent
The claim is proven  using  duality. By scaling, we can assume $\size(B)=1$, which simplifies the notation. We  choose $g \in L^{q'}, \|g\|_{q'}=1$, so that $\|(\mathsf{SB}_{\RR'})^2\|_q=\l  (\mathsf{SB}_{\RR'})^2,g \r$, and set
$$
\RR''=\big\{R \in \RR': \textstyle  \E_\RR g \geq 2^{\frac{1}{q'}}K_q |\sh( \RR')|^{-\frac{1}{q'}}\big\}.$$
By the bound on $\mathsf{M}_\RR$, we have that
$$
 |\sh(\RR'')| \leq \frac12 |\sh(\RR')| \textrm{ and }  \forall \, R \in \RR'\backslash \RR'' \textstyle,\, \E_\RR g  <  2^{\frac{1}{q'}}K|\sh( \RR')|^{-\frac{1}{q'}}.
$$
Therefore
\begin{align*}
\|(\mathsf{SB}_{\RR'})^2\|_q&=  \l  (\mathsf{SB}_{\RR'})^2,g \r \\ &=
\Big(\sum_{ R \in \RR'\backslash \RR''}|b_R|^2   \E_\RR g \Big) + \|(\mathsf{SB}_{\RR''})^2\|_q\|g\|_{q'} \\ & \leq \Big(\sup_{R \in  \RR'\backslash \RR''}   \E_\RR g \Big)\Big(\sum_{R\in \RR'\backslash \RR''}|b_R|^2    \Big) + \|(\mathsf{SB}_{\RR''})^2)\|_q
\\ & \leq\textstyle CK_q  |\sh(\RR')|^{1-\frac{1}{q'}}  + \| (\mathsf{SB}_{\RR'})^2\|_q,
\end{align*}
and this completes the proof of the claim, and in turn, of Lemma \ref{JNSBL}.
\end{proof}
\begin{proof}[Proof of Lemma \ref{B-SBL}]  The lemma is an immediate consequence of Proposition \ref{f-Delta} applied to $f= \mathsf{B}_\RR$, and of the pointwise
bound
\begin{equation} \label{PWB}
(\Delta_{12} \mathsf{B}_\RR) (x) \leq C \mathsf{SB}_\RR(x), \qquad x \in \sh(\RR).
\end{equation}
Let us prove \eqref{PWB}, by computing $\Delta_{12} \mathsf{B}_\RR$ explicitly. Observe that \begin{equation} \label{zeroQR}  Q \not\subset R \implies  \l\psi_R, h_Q \r=0.  \end{equation}
Indeed, if $Q=I_Q \times J_Q $ is not contained in $R=I_R \times J_R$, and $Q\cap R \neq \emptyset$, we have that either $I_R \subsetneq I_Q$, or $J_R \subsetneq J_Q$ (or both). Then \eqref{zeroQR} is a consequence of $\psi_R$   having mean zero (see \eqref{meanzero}), and $h_Q$ being constant, along each horizontal (resp.\ vertical) line. Moreover, again as a consequence of \eqref{meanzero} and of the smoothness of $\psi_R$,
\begin{equation}  Q \subset R \implies   |\l\psi_R, h_Q \r| \lesssim  \Big(\frac{|Q|}{|R|}\Big)^{\frac32}.\label{estQR}
\end{equation}
Therefore \begin{align*} [(\Delta_{12} \mathsf{B}_R)(x)] ^2 & = \sum_{Q \in \D\times \D} \bigg( \sum_{R \in \RR: R\supseteq Q}b_R  \l\psi_R, h_Q \r   \bigg)^2 \frac{\cic{1}_Q(x)}{|Q|} \\ & \lesssim  \sum_{Q \in \D\times \D} \bigg( \sum_{R \in \RR: R\supseteq Q}\frac{b_R}{\sqrt{|R|}} \frac{|Q|}{|R|} \bigg)^2  \cic{1}_Q (x) \\ & \lesssim \sum_{Q \in \D\times \D} \bigg( \sum_{R \in \RR: R\supseteq Q}\frac{|b_R|^2}{ |R|} \frac{|Q|}{|R|} \bigg) \bigg( \sum_{R \in \RR: R\supseteq Q}  \frac{|Q|}{|R|} \bigg) \cic{1}_Q(x) \\ & \lesssim \sum_{R \in \RR} \frac{|b_R|^2}{|R|} \bigg( \sum_{k,\ell\geq 0}  \sum_{\substack{  Q\subseteq R\\ |I_Q| = 2^{-k} |I_R|\\ |J_Q| = 2^{-\ell} |J_R| } } 2^{-k} 2^{-\ell}\cic{1}_Q(x) \bigg) \lesssim \sum_{R \in \RR} \frac{|b_R|^2}{|R|} \cic{1}_R(x) = [\mathsf{SB}_{\RR}(x)]^2,
\end{align*}
which is exactly \eqref{PWB}.
\end{proof}
\subsection{Proof of Proposition \ref{f-Delta}}
\label{ss45} We follow the approach of \cite{PF}. The bound \eqref{f-Delta1} is obtained as a consequence of the vector-valued inequality for the one parameter square function
\begin{equation}   \label{f-Delta2}
\big\|\| ( g_I )\|_{\ell^2(I)}\big\|_{L^p_{x_2}} \leq   Cp^{\frac12} \|\|  \Delta_2 g_I \|_{\ell^2(I)}\big\|_{L^p_{x_2}}, \qquad 2\leq p<\infty.
\end{equation}
Consider the $\ell^2$-valued function
$
F_I(x_1,x_2) = \l f(\cdot,x_2), h_I \r h_I(x_1).
$
Then $\Delta_{12} f = \|\Delta_2 F_I\|_{\ell^2(I)} $,  so that
\begin{align*}
\|\Delta_{12} f\|_p & = \big\|\| (\Delta_2 F_I)\|_{\ell^2(I)}\big\|_{p} = \Big\| \big\|\|  \Delta_2 F_I (x_1,\cdot)\|_{\ell^2(I)}\big\|_{L^p_{x_2}} \Big\|_{L^p_{x_1}} \\ & \geq C^{-1}p^{-\frac12} \Big\| \big\|\|  F_I (x_1,\cdot)\|_{\ell^2(I)}\big\|_{L^p_{x_2}} \Big\|_{L^p_{x_1}} =  C^{-1}p^{-\frac12} \Big\| \big\|\|  F_I (x_1,\cdot)\|_{\ell^2(I)}\big\|_{L^p_{x_1}} \Big\|_{L^p_{x_2}} \\ & =
C^{-1}p^{-\frac12} \Big\| \big\|\Delta_1 f(\cdot,x_2)\big\|_{L^p_{x_1}} \Big\|_{L^p_{x_2}} \geq C^{-2}p^{-1} \Big\| \big\|  f(\cdot,x_2)\big\|_{L^p_{x_1}} \Big\|_{L^p_{x_2}}   =C^{-2}p^{-1}\|f\|_p.
\end{align*}
We used the vector-valued bound \eqref{f-Delta2} to get to the second line, and the scalar valued bound to obtain the last inequality. We have thus proved Proposition \ref{f-Delta}, up to showing that \eqref{f-Delta2} holds.  This is done by applying the Rubio De Francia trick to the weighted inequality
\begin{equation}   \label{f-Delta3}
\int_{\R}\|f_k\|_{\ell^2(k)}^2 w(x_2) \, \d x_2 \leq C [w]_{A_1}  \int_{\R}\|\Delta_2 f_k\|_{\ell^2(k)}^2 w(x_2) \, \d x_2,
\end{equation}
which is obviously equivalent to the scalar case. In turn, the scalar case of \eqref{f-Delta3} follows from the Chang-Wilson-Wolff good-$\gamma$ inequality \cite{CWW}
$$
\big|\big\{ x_2: \sup_{x_ 2 \in I}\textstyle  |\E_I f|>2\lambda, \Delta_2 f(x_2) \leq \gamma \lambda\big\} \big| \leq C\e^{-\frac{c}{\gamma^2}}\big|\big\{ x_2: \displaystyle\sup_{x_ 2 \in I}\textstyle  |\E_I f| > \lambda \big\} \big|.
$$
One chooses $\frac{1}{\gamma^2} =c  [w]_{A_1}$ and recalls the following property of $A_1$ weights: there exists $c>0$ such that
$$
E\subset I , \, |E| \leq \e^{-c[w]_{A_1}} |I| \implies w(E) \leq \textstyle \frac{1}{8}w(Q).
$$
This yields a weighted good-lambda inequality whose integration gives (the scalar case of) \eqref{f-Delta3}. See \cite{PF} for details.
\subsection{Dropping the compact support assumption} \label{ss46}  Let   $\RR$ be as above, and  now  let   $\{\psi_R\}_{R \in \RR}$  be product  wave packets $L^2$- adapted to $\RR$ for which we no longer assume \eqref{supp}. Let us prove that \eqref{JNp} continues to hold. The main tool
is the following lemma, which is a variant of \cite[Lemma 3.1]{MPTT}.
\begin{lemma} \label{getcptsp}
Let $\psi_R$ be  a function $L^2$ adapted to the rectangle $R$  and  having mean zero along each line parallel to the coordinate axes of $R$, as in \eqref{meanzero}.
 Then
\begin{equation} \label{sect6-cpteq} \psi_R(x_1,x_2)=\sum_{k\geq 0} 2^{-100k} \psi_{R^{(k)}}(x_1,x_2),
\end{equation}
 whre for each $k \geq 0$,   $\{\psi_{R^{(k)}}\}$ is  a collection of product wave packets adapted to $\{R^{(k)}: R \in \RR\}$, with $\mathrm{supp}\, \psi_{R^{(k)}} \subset R^{(k)}.$ The adaptation constants depend  only on the adaptation constants of $\psi_R$ (in particular, they do not depend on $k$). \end{lemma} \noindent
We apply Lemma \ref{getcptsp} to each $\psi_R$, so that
\begin{equation} \label{getcpt1}
\mathsf{B}_\RR (x) = \sum_{R \in \RR} b_R \psi_R(x) = \sum_{k\geq 0} 2^{-100k} \sum_{\rho \in \RR^{(k)}} b_\rho \psi_\rho(x)= \sum_{k\geq 0} 2^{-100k} \mathsf{B}^{(k)}_{\RR^{(k)}}(x).
\end{equation}
where $\RR^{(k)}:=\{\rho=R^{(k)}: R \in \RR\}$ and we defined $B^{(k)}=\{b_\rho:=b_R: \rho=R^{(k)} \in \RR^{(k)}\}$.
Then, it is easy to see that $\size(B^{(k)}) \leq\size(B).$ As discussed in Subsection 4.1, $\RR^{(k)}$ can be split into $Ck^22^{2k}$ subcollections of rectangles whose horizontal and vertical sides fit into  a single dyadic grid. Thus, applying \eqref{JNp} for each of these subsets,
$$
\big\|\mathsf{B}^{(k)}_{\RR^{(k)}}\big\|_p \leq Ck^22^{2k}p^2  \size(B^{(k)})|\sh(\RR^{k})|^{\frac 1p}.
$$
Finally, using \eqref{getcpt1}, we conclude that \begin{align} \label{JNp-cpt}
\|\mathsf{B}_\RR\|_p &\leq \sum_{k\geq 0} 2^{-100k} \big\|\mathsf{B}^{(k)}_{\RR^{(k)}}\big\|_p \leq Cp^2 \Big(\sum_{k\geq 0} 2^{-90k} \size(B^{(k)})|\sh(\RR^{(k)})|^{\frac 1p} \Big) \\&\leq C p^2\size(B )|\sh(\RR )|^{\frac 1p}, \nonumber
\end{align}
as claimed.

\section{The time-frequency phase plane}
We briefly recall (see \cite{LL1}) how  to produce a discretization of the multiplier operator
$$
\widehat{G_\alpha f}(|\xi|\e^{2\pi i \theta(\xi)}) =   \hat f(\xi) \cic{1}_{\alpha}(\theta(\xi)).
$$
where $\alpha=(0, \alpha)$ is a subinterval of $(0,1)$.

Following \cite{LL1}, we start by decomposing the frequency plane into a union of annular sectors.
For $\ann \in 2^{\mathbb Z}$ and $
\omega \in \D$, $\omega\in [0,1]
$ define $$
 \Omega_{\ann,\omega}= \big\{\xi=|\xi|\e^{2\pi i \theta}: |\xi| \in [\textstyle \frac34\ann,\frac74\ann), \theta \in \omega \big\}.  $$
and
$$\S_{\ann,\omega}=\{R\times\Omega_{\ann,\omega}:R\in \mathcal R_{\ann,\omega}\}.$$
Here $\mathcal R_{\ann,\omega}$ denotes all the rectangles with dimensions $\ann^{-1}$ and $|\omega|^{-1}|\ann|^{-1}$, which are obtained by rotating with angle $c(\omega)$ the rectangles with the same dimensions from $\D\times\D$. The elements in the collection
$$\S_u=\bigcup_{\omega,\ann}\S_{\ann,\omega}$$
will be referred to as {\em tiles}. We will typically write a tile $s$ as $s=R_s\times \Omega_s$, and also denote by $\ann(s)$, $\omega(s)$ the corresponding components. Also, $\omega_{1s}, \omega_{2s}$ will be the left and right dyadic children of $\omega_s$, while $\Omega_{1s},\Omega_{2s}$ will denote the corresponding sectors. The number $\ecc(s):=|\omega_s|$ will be referred to as the eccentricity of $s$.

Let $\Phi$   be a positive Schwartz function supported on $[\frac34,\frac74]$ with
$$
\sum_{\ann} \Phi_{\ann} = 1,
$$
where $\Phi_{\ann}(t)=\Phi(\ann^{-1}t)$.

The next step is the partition of unity for $\cic{1}_{\alpha}$ described in Chapter 7 of \cite{THWPA}. Following this approach, the analysis of $\widehat{G_\alpha f}$ is reduced to that of operators of the form
$$\sum_{\omega\subset [0,1]:\omega\in \G}\hat f(\xi)\beta_{\omega}(\theta(\xi))1_{\omega_2}(\alpha).$$
Here $\G$ is a fixed dyadic grid, independent of $\alpha$. For all practical purposes, we can assume $\G$ is in fact the standard dyadic grid $\D$. We will denote by $\omega_1,\omega_2$ the left and right dyadic children of $\omega\in\D$. Also,
 each $\beta_\omega$ is a smooth bump function supported on $\omega_1$ such that
\begin{equation}
\label{unifad}
\|\beta_\omega^{(j)}\|_\infty \leq C_j |\omega|^{-j}, \qquad j \geq 0.
\end{equation}

Finally, note that the multiplier
$$m_{\ann,\omega}(\xi)=\beta_{\omega}(\theta(\xi))\Phi_{\ann}(|\xi|)$$
is supported in and $L^\infty$ adapted to the rectangle-like region $\Omega_{\ann,\omega_1}$. By applying a standard windowed Fourier series we can write
\begin{equation}
\label{reprf1}
\hat f(\xi) m_{\ann,\omega}(\xi)= \sum_{s\in\S_{\ann,\omega} } {\l f, \varphi_{s} \r} \hat  \psi_{s} (\xi),
\end{equation}
where for each $s=R_s\times \Omega_s\in \S_{\ann,\omega}$,
 $ \varphi_{s}$ and $ \psi_{s}$ are Schwartz functions supported in frequency in $\Omega_{\ann,\omega_{1s}}$ and  $L^2$- adapted to the rectangle $R_{s}$.

By putting all these things together, it follows that we can write $G_\alpha f(x)$ as a combination of  model sums of the form
\begin{equation}
\label{modelroughctf}
\sum_{s\in\S_u } {\l f, \varphi_{s} \r} \psi_{s} (x)1_{\omega_{2s}}(\alpha).
\end{equation}

By splitting the model sum in two parts, we can further assume that $\ann(s)\in 4^{\Z}$ for each $s\in\S_u$. In particular, if $\ann(s)\not=\ann(s')$ then $\langle\varphi_{s},\varphi_{s'}\rangle=\langle\psi_{s},\psi_{s'}\rangle=0$

For the purpose of future applications of the theory in Section \ref{Grids}, we remark that $ \varphi_{s}$ and $ \psi_{s}$ are in fact {\em product} wave packets adapted to $R_s$. This follows since the canonical directions $e^{2\pi ic(\omega_s)},e^{i\frac\pi 2+2\pi ic(\omega_s)}$ of $R_s$ do not intersect the frequency support $\Omega_{1s}$. The same argument shows that $ \varphi_{s}$ and $ \psi_{s}$ continue to have the same property if $R_s$ is slightly rotated so that the canonical directions become $v,v^\perp$, for any  $v\in\omega_{2s}$.

Finally, we mention that the discretization \eqref{modelroughctf} will serve us well when we treat the operator $H_\V$ in the following sections. The fact that we have a  sharp cutoff $1_{\omega_{2s}}(\alpha)$ in the model sum will be absolutely crucial in our argument. We do not know if this decomposition can also be achieved for general kernels $K$ (that is for the operators $T_\V$). All earlier papers on the vector field problem used model sums involving smooth cutoffs $\beta_{\omega_{2s}}(\alpha)$ ($\beta$ smooth), which can in turn be realized as building blocks for all $T_\V$. The fact that our argument needs rough cutoffs is due to our use of product trees. In short, any smooth model for a product tree will produce error terms that are too large to be considered negligible.

\section{Model sums, trees, and an $L^p$ almost orthogonality principle}

\subsection{Trees, size, and an almost orthogonality principle} Let $\S$ denote a fixed arbitrary finite subset of $\S_u$.
We call \emph{trees} those subsets of $\S_u$ which point roughly in the same direction.
\subsubsection{Trees and size}A subset $\t \subset \S_u$ is a \emph{lacunary tree} if
\begin{equation}
\label{lactree}
\omega_{2\t} = \bigcap_{s \in \t } \omega_{2s} \neq \emptyset.
\end{equation}
Thus, for a lacunary tree $\t$, the Heisenberg boxes $R_s\times \Omega_{1s}$ of  $\varphi_s$ are  pairwise disjoint, and it is easy to see that
  we have the almost orthogonality relation
\begin{equation}
\label{aorel}
\sum_{s \in \t} |\l f, \varphi_s \r|^2 \lesssim \|f\|_2^2.
\end{equation}
 A subset $\t \subset \S_u$ is an \emph{overlapping tree} if
\begin{equation}
\label{ovtree}
\omega_{1\t} = \bigcap_{s \in \t } \omega_{1s} \neq \emptyset.
\end{equation}
For an  overlapping tree $\t$, the directional support uncertainty intervals $\{\omega_{2s}: s \in \t\}$ are  pairwise disjoint.

A subset $\ct \subset \S$ is a  \emph{conical tree} if
$
\omega_{2s}=\omega_{2s'}=:\omega_\ct,$ for all $ s, s' \in \ct.
$
A conical tree is both lacunary and overlapping.
Any $v_\t \in \omega_{2\t}$ (resp.\ $v_\t \in \omega_{1\t}$, $v_\ct \in \omega_{ \ct}$) is a \emph{top direction} of the lacunary (resp.\ overlapping, conical) tree $\t$ (resp. $\ct$). If $\t$ is a lacunary tree and  $s \in \t$, then  \begin{equation} \label{meanzero-tree}\int_{\{x+tv_\t : t \in \R\}} \varphi_{s}   =  \int_{\{x+tv_\t^\perp: t \in \R\}} \varphi_{s}     = 0, \; \forall \, x \in \R^2.
\end{equation}
 The \emph{shadow} and $k$-\emph{shadow} of a tree are respectively  defined as
$$
\sh(\t) = \cup\{R_{s}: s \in \t\}, \qquad \sh_k (\t)= \cup\{R_{s}^{(k)}: s \in \t\},
$$
where $R_s^{(k)}=2^kR_s.$ For future use, we also introduce the notation
\begin{equation}
\label{overhang} \mathsf{crown}(\t)= \bigcup_{s \in \t} \omega_{2s} \subset (0,1) \sim S^1
\end{equation}
for the \emph{crown} of a tree $\t$.

Let $f$ be a given function.
The \emph{lacunary size} of $\S'\subset \S$ (with respect to $f$, whose dependence in the notation is suppressed) is given by
\begin{equation}
\label{lacsize}
\size(\S') = \sup_{\substack{\t \subset \S'\\ \textrm{lacunary tree}}} \Big( \frac{1}{|\sh(\t)|} \sum_{s \in \t} |\l f, \varphi_s\r|^2 \Big)
\end{equation}
where $\varphi_s$ are product  wave packets adapted to $\S$.
The  \emph{conical size} of $\S'\subset \S$ is defined as
\begin{equation}
\label{csize}
\size_{\nabla}(\S') = \sup_{\substack{\ct \subset \S'\\ \textrm{conical tree}}} \Big( \frac{1}{|\sh(\ct)|} \sum_{s \in \ct} |\l f, \varphi_s\r|^2 \Big).
\end{equation}
Since conical trees are in particular lacunary trees, we always have $\size_{\nabla}(\S') \leq \size(\S') $.
We recall for future use \cite[Lemma 4.40]{LM}
\begin{equation}
\label{remsize}
\Big( \frac{1}{|\sh(\t)|} \sum_{s \in \t} |\l f, \varphi_s\r|^2 \Big)^{1/2} \lesssim \|f\|_{\infty},
\end{equation}
for each lacunary $\t$.

\subsubsection{Single tree operators} Let  $\t$ be a lacunary tree. Let $\{\varphi_s,\psi_s\}_{s \in \t}$, and for some $k \geq 0$, $\{\psi_{s}^{(k)}\}_{s \in \t} $, be product wave packets adapted to $\{R_s\}_{s \in \t}$ (resp.\ $\{R_s^{(k)}\}_{s \in \t}$). Assume in addition that $\mathrm{supp} \, \psi_{s}^{(k)} \subset R_s^{(k)}.$
Define the model sums
\begin{equation}
\label{st-lacms}
f_\t= \sum_{s \in \t} \l f, \varphi_s \r \psi_s,
\end{equation}
and \begin{equation}
\label{st-lac-k-ms}
f_\t^{(k)}= \sum_{s \in \t} \l f, \varphi_s \r \psi_s^{(k)}.
\end{equation}
\begin{lemma}
\label{st-lac}
We have the estimate
\begin{equation}
\label{st-lac-est}
\|f_\t\|_p \leq C p^2\size(\t), \qquad 2\leq p< \infty.
\end{equation}
\end{lemma}
\begin{lemma}
\label{st-lac-k}
We have the estimate
\begin{equation}
\label{st-lac-k-est}
\big|\big\{x \in \sh_k(\t): |f_\t^{(k)}(x)| > \lambda\big\}\big| \leq C\exp\Big(-c \textstyle \sqrt{\frac{\lambda}{\size(\t)}} \Big)|\sh_k(\t)|.
\end{equation}
\end{lemma}
\begin{proof}[Proofs of Lemmata \ref{st-lac} and \ref{st-lac-k}] Estimates \eqref{st-lac-est} and \eqref{st-lac-k-est}   follow  respectively from \eqref{JNp-cpt} and \eqref{JNexp}. Let us explain the details for \eqref{st-lac-est}. Note that for each $s \in \S$,  $v_\t \in \omega_{2s}$, which is within the uncertainty angle of  $R_s$. Therefore,    property \eqref{meanzero} still holds if we replace $e_{R_s}$ by $v_\t.$  Then, we can find a rectangle $\rho_s$ oriented along $v_\t$ such that $R_s \subset \rho_s \subset 90R_s $ and $\{\rho_s: s \in \t\}$ are products of intervals coming from finitely many dyadic grids. Thus one observes that setting
$  b_{\rho_s}:= \l f,\varphi_s\r,
$ gives $f_\t = \mathsf{B}_{\{\rho_s:s \in \t\}}$, and $\{\varphi_s, \psi_s\}$ are still product wave packets adapted to $\{\rho_s: s \in \t\}$. Therefore \eqref{JNp-cpt} applies.
\end{proof}
\noindent An application of \eqref{JNexp1P} yields a similar result for conical trees. Let $\ct$ be a conical tree. Let  $\{\varphi_s\}_{s \in \ct}$, and for some $k \geq 0$, $\{\psi_{s}^{(k)}\}_{s \in \ct} $, be  wave packets adapted to $\{R_s\}_{s \in \ct}$ (resp.\ $\{R_s^{(k)}\}_{s \in \ct}$).
 Then
\begin{equation}
\label{st-ov-k-est}
\big|\big\{x \in \sh_k(\ct): |f_\ct^{(k)}(x)| > \lambda\big\}\big| \leq C\exp\Big(-c \textstyle \frac{\lambda}{\size_\nabla(B)} \Big)|\sh_k(\ct)|.
\end{equation}

\subsubsection{An $L^p$ almost orthogonality principle and an application to $H_\V$} \label{aoss}
We prove the following $L^p$ almost orthogonality result.
\begin{theorem}
\label{sect2-Lp-ortho}
Let $\cic{\alpha}$ be a collection of disjoint intervals.  We have the estimate
\begin{equation} \label{sect2-Lp-eq1}
\Big(\sum_{\alpha \in \cic{\alpha}}\|G_{\alpha} f \|^p_p \Big)^{\frac1p} \lesssim_p \|f\|_p, \qquad \forall p \geq 2.
\end{equation}
\end{theorem}
\begin{proof} By restricted-type interpolation, it suffices to show that \eqref{sect2-Lp-eq1} holds for $|f|\le\cic{1}_E$, where $E\subset\R^2$ has finite measure.  We want to use a  model sum similar to the one in  \eqref{modelroughctf}. However, we use a more straightforward partition of unity for ${\alpha}=(\alpha_l,\alpha_r)$, namely
$$1_{{\alpha}}=\sum_{\omega\in C_{{\alpha}}}\beta_\omega.$$
Here $C_{{\alpha}}$ consists of all intervals  $\omega\in \D\cup\D^1\cup\D^2$ such that $\omega\subset \alpha$ and  $\dist(\omega,\{\alpha_l,\alpha_r\})\sim |\omega|$, while $\beta_\omega$ is supported in and $L^\infty$- adapted to $\omega$. Since $\|\sum_{\omega\in C_{{\alpha}}}1_\omega\|_{\infty}\lesssim 1$, by working with a finite number of subsets we can assume the intervals in $C_{{\alpha}}$ are pairwise disjoint, $\dist(\omega,\alpha_l)\sim |\omega|$ and prove
\begin{equation} \label{sect2-orthpf-fine}
\Big(\sum_{\alpha \in \cic{\alpha}}\|f_{\t(\alpha)}\|^p_p \Big)^{\frac1p} \lesssim_p \|f\|_p, \qquad \forall p > 2.
\end{equation}
Here
$$\t(\alpha)=\{s\in\S_u:\omega_s\in C_\alpha\},$$
$$f_{\t(\alpha)}=\sum_{\omega_s\in \t_\alpha}\langle f,\varphi_s\rangle \psi_s,$$
and for each $s=R_s\times \Omega_s$,  $\varphi_s,\psi_s$ are product wave packets adapted to $R_s$ and supported in frequency in $\Omega_s$. The axes of $R_s$ are oriented along $e^{2\pi i\alpha_l}, e^{\frac\pi 2+2\pi i\alpha_l}$

Note that $\t(\alpha)$ is similar in nature to a lacunary tree, in particular \eqref{aorel}, \eqref{remsize} and  \eqref{st-lac-est} will also hold for it.
We   chop each  $\t(\alpha)$ into subsets $\t_\sigma(\alpha)$  by means of the following standard iterative procedure.
\begin{itemize}
\item \textsf{INIT} $\mathrm{Stock}:= \t(\alpha)$, $\sigma\sim1$;
\item \textsf{WHILE} $\mathrm{Stock}\neq \emptyset$, select the largest collection $\t_\sigma(\alpha) \subset \mathrm{Stock}
$ with
$$
\sum_{s \in \t_\sigma(\alpha) } |\l f,\varphi_s\r|^2  \geq{\textstyle\frac{ \sigma^{2}}{4}} |\sh(\t_\sigma(\alpha))|.
$$
Set $\mathrm{Stock}:=\mathrm{Stock}\backslash \t_\sigma(\alpha)$; $\sigma:=\frac\sigma2$.
\end{itemize}
This iterative process produces  the decomposition
\begin{equation} \label{sect2-orthpf3}
\t( \alpha)= \bigcup_{\sigma \lesssim 1} \t_\sigma(\alpha), \qquad  \size(\t_\sigma(\alpha)) \leq \sigma,
\end{equation}
and the almost-orthogonality \eqref{aorel} of the $\{\varphi_s:s \in \t(\alpha)\}$ yields
\begin{equation} \label{sect2-orthpf4}
|\sh(\t_\sigma(\alpha))| \lesssim \sigma^{-2}\sum_{s \in \t_\sigma(\alpha) } |\l f,\varphi_s\r|^2 \lesssim \sigma^{-2 } \|G_\alpha f\|_2^2.
\end{equation}
Therefore
\begin{align*}
\Big(  \sum_{\alpha \in \cic{\alpha}} \big\|f_{\t(\alpha)} \big\|_p^p \Big)^{\frac1p} & = \Big\|\|f_{\t(\alpha)}\|_{\ell^p(\cic{\alpha})}\Big\|_p = \Big\|\Big\|\sum_{\sigma \lesssim 1}f_{\t_\sigma(\alpha)} \Big\|_{\ell^p(\cic{\omega})}\Big\|_p \\ & \leq \sum_{\sigma\lesssim 1} \Big(\sum_{\alpha \in \cic{\alpha}} \big\| f_{\t_\sigma(\alpha)}\big\|_p^p \Big)^{\frac1p} \lesssim  \sum_{\sigma \lesssim 1} \sigma\Big(\sum_{\alpha\in \cic{\alpha}} |\sh( \t_\sigma(\alpha))| \Big)^{\frac1p} \\ & \lesssim \sum_{\sigma \lesssim 1} \sigma ^{1- \frac{2}{p} }\Big( \sum_{\alpha \in \cic{\alpha}} \|G_\alpha f\|_2^2 \Big)^{\frac1p} \lesssim {(\|f\|_2^2)^\frac{1}{p}}= |E|^{\frac1p}.
\end{align*}
In the above display, we used  the single tree estimate \eqref{st-lac-est} of Lemma \ref{st-lac} to get the second inequality of the second line, \eqref{sect2-orthpf4} in going from the second to the third line, and disjointness of $\alpha \in \cic{\alpha}$ to conclude. This proves  \eqref{sect2-orthpf-fine}, and in turn, Theorem \ref{sect2-Lp-ortho}.
\end{proof}
We can use the $L^p$-orthogonality principle of Theorem \ref{sect2-Lp-ortho} to give a simple proof of the bound
$$\|H_\V f\|_p \lesssim_p \log N \|f\|_p, \qquad 2<p<\infty.$$
The proof is done by induction on $n$, where $N=2^n$. For the inductive step, assume that whenever $\# \V' \leq 2^{n-1}$,  $\|H_{\V'}\|_{p \to p} \leq C_p( n-1)$, where $C_p$ is bigger than the implicit constant appearing in Theorem \ref{sect2-Lp-ortho}. Order $\V$ increasingly and let $\V'=\{v_i \in \V: i$ odd$\}$. For $i$ odd, let $\omega_i$ be the frequency cone $\big\{ \xi: \frac{\xi}{|\xi|} \in \frac\pi 2 +[v_{i},v_{i+2}) \big\}.$ For each $i$ odd,
\begin{equation} \label{Lp-RM:1}
v\in\{v_i ,v_{i+1}\} \implies |H_v f(x)| \leq |H_{v_{i}} f(x)| + |H_{v_{i+1}} [G_{\omega_i} f](x)|,
\end{equation} and taking suprema and $L^p$ norms,
\begin{align*}
\| H_{\V} f\|_p &\leq  \| H_{\V'} f\|_p + \Big\|\sup_{i \textrm{ odd}} H_{v_{i+1}} [G_{\omega_i} f] \Big\|_p \\
& \leq C_p (n-1)\|f\|_p + \Big(\sum_{i} \|  G_{\omega_i} f\|^p_p \Big)^{\frac1p}\\
& \leq C_p n\|f\|_p,
\end{align*}
which completes the induction. The last line is an application of Theorem \ref{sect2-Lp-ortho}, while, in going from the first to the second line, we have used the induction hypothesis and the fact that only one $i$ contributes at each point.

\subsection{The model sum for $H_\V$}
We recall that, for a direction $v$, the directional multiplier $H_v$ is given by
$
\widehat{H_v f}(\xi)= \hat{f}(\xi) \mathrm{sign}(v\cdot \xi),
$
which is essentially the same as
$$
\hat f(\xi) \cic{1}_{(0,\tau(v))} (\theta), \qquad \xi=|\xi|\e^{2\pi i \theta},
$$
where $v=\exp\big({2\pi i \tau(v)- \frac{\pi i}{2}}\big)$.
Then, we use \eqref{modelroughctf} to   define  the model sum for $H_v$
\begin{equation} \label{Hv-model}
\Hh_{ \S} f(x, v) = \sum_{s \in\S } \l f,\varphi_s \r \psi_s(x)\cic{1}_{\omega_{2s}} (v) ,
\end{equation}
where $\S$ is any finite subcollection of $\S_u$, and we identify  $\tau(v)$ with $v$.
Accordingly, the model operator for the maximal multiplier $H_\V$ is defined as\begin{equation}
\label{sect6-main}
    \Hh_\S^\star f(x) = \sup_{v \in \V} \big|\Hh_\S f(x,v)\big|.
\end{equation}

Finally, we write $\Hh_\S^\star$ as a rapidly decaying sum of pieces supported in the dilates $\sh_k(\S)$, $k \geq 0$.
Applying Lemma \ref{getcptsp} to each $\psi_{s}$, we write
$   \psi_s =\sum_{k\geq 0} 2^{-100k} \psi_{s}^{(k)},
$ where  $\psi_{s^{(k)}}$ are product wave packets adapted to the rectangles $\{ R_{s}^{(k)}:=2^k R_s\}$, with the mean zero property
\begin{equation} \label{meanzero-varg}  v\in \omega_{2s} \implies \int_{\{x+tv : t \in \R\}} \psi_{s}^{(k)}  =  \int_{\{x+tv^\perp: t \in \R\}} \psi_{s}^{(k)}    = 0, \; \forall \, x \in \R^2.
\end{equation} We thus have
$$
\Hh_{\S} f (x,v) = \sum_{k\geq 0}2^{-100k}\sum_{s \in\S  } \l f,\varphi_s \r  \psi_{s}^{(k)}(x)\cic{1}_{v \in \omega_{2s} }
$$
 Defining
$$
\Hh_{\S^{(k)}}^\star f (x) = \sup_{v \in \V} \Big| \sum_{s \in\S  } \l f,\varphi_s \r  \psi_{s}^{(k)}(x)\cic{1}_{v \in \omega_{2s}}\Big|,
$$
we have
the pointwise inequality
$$
\big|\Hh_\S^\star f(x)\big| \lesssim \sum_{k\geq 0} 2^{-100k} \big|\Hh_{\S,k}^\star f(x)\big|.
$$In particular, we reduced   Theorem 3 to the proof of  the weak-type bound
\begin{equation}
\label{sect6-main2}
\|\Hh_{\S^{(k)}}^\star f\|_{2,\infty} \lesssim_Q 2^{4k}\sqrt{\log N}({\log\log N})^6.
\end{equation}

\section{Vargas type $\V$: Proof of Theorem 3}
Before entering the proof, we begin the section with a more detailed description of Vargas  sets of directions, and provide some important examples.

\subsection{Vargas sets} \label{Vargas-ss}
We  recall Definition \ref{vargas-def}. We say that a set of $N$ directions is a Vargas set  with constant $Q$ if the longest lacunary sequence (see Definition \ref{lacunary-def}) in $\V$  has $Q\log N$ elements.  Our definition is easily seen to be equivalent to the one of  \cite{KATZ2},  where the terminology has been introduced.
For example, uniformly distributed  sets
$$
\V_N= \big\{v_j= {\textstyle \frac{j}{2\pi N}}, j=0, \ldots, N-1\big\}
$$
are of (uniform in $N$) Vargas sets with constant $1$.
 Another prime example of (uniformly) Vargas sets are finite truncations of Cantor-type sets.
\vskip1mm \noindent \textsc{Claim.} The
 truncated $q$-adic ($q \geq 3$) Cantor sets of $N=2^n$ elements
$$
\V_N:=\mathfrak{C}_{q,n}=\bigg\{ \sum_{j=1}^n a_j q^{-j},
\; (a_1,\ldots,a_n)\in\{0,q-1\}^n\bigg\},
$$
are   Vargas sets with constant $4$ (independent of $q$ and $N$).
\begin{proof} We argue by induction on $n$.  The case $n=1$ is trivial.   The inductive assumption is that  for each $\nu<n$, a lacunary sequence contained in $\mathfrak{C}_{q,\nu}$ has at most $4\nu$ elements. Take a lacunary sequence
  $$\Big\{v_\ell=\sum_{j=1}^n v_\ell^{(j)} q^{-j}, \; \ell=1,\ldots,r\Big\} \subset \mathfrak{C}_{q,n}$$ Let $h$ be the first integer such that $v_1^{(h)} \neq v_2^{(h)}$. Then $$ \textstyle
\frac{q-1}{q^h}\leq |v_1-v_2| \leq \frac{1}{q^{h-1}}. $$
As a consequence of the lacunarity property \eqref{lacunary-defeq}, it is easy to see that
$$
|v_\ell -v_m| \leq \textstyle \frac{3}{2^{m-2}} |v_{2}-v_1|, \qquad \forall\, m \geq 2, \;  \ell=m, \ldots, r.
$$
 Hence,  for each $\ell=5,\ldots, r$, we must have $v_\ell^{(j)}=v_5^{(j)}$, for each $j=1,\ldots,h$. If not, there would be   some $j\leq h, \ell\geq 5$ such that  $$ \textstyle
v^{(j)}_\ell\neq v^{(j)}_5 \implies \frac{q-1}{q^{j}}\leq |v_\ell- v_5| \leq \frac 38  |v_1-v_2| \leq \frac 38\frac{1}{q^{h-1}}, $$
which gives the contradiction $q-1\leq \frac38 q$. This shows that  $\{v_5,\ldots,v_r\}$ is a lacunary sequence  contained in
$$
  \left\{ \sum_{j=1}^n a^{(j)} q^{-j},
\; a^{(1)}=v_5^{(1)},\ldots,a^{(h)}=v_5^{(h)},(a^{(h+1)},\ldots,a^{(n)})\in\{0,q-1\}^{n-h}\right\} \sim \mathfrak{C}_{q,n-h},
$$
so that $r-4\leq 4(n-h)$ by the inductive assumption. We get $r \leq 4n-4h+4 \leq 4n$, which completes the inductive step.
\end{proof}
  The claim below shows how Vargas sets and lacunary sets   are, in a sense, extremes.
\vskip1mm \noindent \textsc{Claim.} Let  $\#\V=N$.    Then $\V$ contains a lacunary sequence of $\frac{\log N}{3}$ elements.
\begin{proof} This is best seen by induction on $n=\frac{\log N}{3}$.  Take as base case $n=4$. The inductive assumption is that every set of at least $2^{3\nu}$ directions contains a lacunary sequence of $\nu$ elements, for $4 \leq \nu\leq n$. Let $\V=\{v_1 <v_2<\ldots v_{2^{3(n+1)}}\}$ be given. We can rescale and assume that $v_1=0,$ $v_{2^{3(n+1)}}=1$. Define $\V_j = \V \cap [\frac{j-1}{2^3}, \frac{j}{2^3}]$, $j=1,\ldots,8$. By pigeonholing, at least one $\V_j$ has at least $2^{3n}$ elements: assume it is so for $\V_2$ (the other cases are treated similarly). Then, by inductive assumption, $\V_2$ contains a lacunary sequence  $\{w_1, \ldots, w_n\} \subset [\frac{1}{8},\frac{1}{4}]$ with node $w$.
Moreover, by repeated application of \eqref{lacunary-defeq},
$\mathrm{dist}(w,\V_2) \leq 2^{-(n-1)}$
so  that
\begin{equation}
\label{vargas-1} \textstyle
|w-w_{1}| \leq \mathrm{diam}(\V_2) + \mathrm{dist}(w,\V_2) \leq \frac{1}{8}+ 2^{-(n-1)} \leq \frac 14.
\end{equation}
 We set $w_{0}:=v_{2^{3(n+1)}}=1$. Then $w \in [0,\frac12]$, so that $|w-w_{n+1}| \geq \frac12$. Comparing this to \eqref{vargas-1}, we observe that $\{w_0,w_1, \ldots, w_{n}\}$ is a lacunary sequence of $n+1$ elements contained in $\V$. The inductive step is complete.
\end{proof}

\subsection{Preliminary reductions and outline of the proof}
 Let $\V$ be a fixed set of  directions  with $\#\V=N$ and Vargas constant $Q$.
Using the notations and the results of Section 7, Theorem \ref{main-ub3} follows from the weak-type bound
$$
\|\Hh_{\S^{(k)}}^\star f\|_{2,\infty} \lesssim_Q 2^{4k}\sqrt{\log N}({\log\log N})^6.
$$ From now on, the dependence on $Q$ of the estimates is kept implicit.
By  scaling $f$, it will suffice to show that
\begin{equation}
\label{sect63-1}
\big| \big\{ x \in \R^2 : \Hh_{\S,k}^\star f(x) \gtrsim 2^{2k} \sqrt{\log N}({\log\log N})^6
 \big\} \big| \lesssim 2^{4k} \|f\|_2^2.
\end{equation}
We will obtain \eqref{sect63-1} by constructing a set $A=A^1 \cup A^2\subset \R^2$ with the properties
\begin{equation} \label{sect63-1*}
|A| \lesssim 2^{4k}\|f\|^2_2, \qquad \Hh_{\S^{(k)}}^\star f(x) \lesssim 2^{2k} \sqrt{\log N} (\log \log N)^6 \quad \forall x \not\in A.
\end{equation}
This will be done by decomposing $\S$ into the disjoint union of forests (i.e.\ collections of trees) $\S_\sigma$, where $\sigma$ is a dyadic parameter ranging in $\{z \in 2^{\mathbb Z}: z \leq 2^K\} $, which stands for the size of the subcollection $\S_\sigma$, as we will detail later. Here, $K =\size(\S)$.  Accordingly, using the triangle inequality
\begin{equation}
\label{splitsize}
\Hh_{\S^{(k)}}^\star f(x) \leq \sum_{  \sigma \leq 2^K} \Hh_{\S_\sigma^{(k)}}^\star f (x),  \end{equation}
and each piece $\Hh_{\S_\sigma^{(k)}}^\star f$ is supported on $\sh_k(\S_\sigma)$. In turn, each $\S_\sigma$ will be the union of trees $\t \in \F_\sigma$ with the properties
\begin{align}
\label{bessel-prel}
&\sum_{\t \in \F_{\sigma}} |\sh(\t)| \lesssim \sigma^{-2}\|f\|_2^2,\\
\label{count-prel} &
\sup_{v \in \V}\sum_{\t \in \F_\sigma} \cic{1}_{\mathsf{crown}(\t)} (v) \leq C\log N.
\end{align}
Property \eqref{count-prel} will be obtained as a consequence of the Vargas set structure of $\V$.
We take the first component of the exceptional set $A$ to be
$$A^1 =\bigcup_{1\leq \sigma \leq 2^K}\bigcup_{\t  \in \F_{\sigma}} \sh_k(\t).
$$
Using \eqref{bessel-prel}, we have
\begin{equation} \label{eset1}
|A^1|  \leq \sum_{1\leq \sigma \leq 2^K}  \sum_{\ct \in \F_\sigma} |\sh_k(\t)|  \lesssim
 2^{4k} \sum_{1\leq \sigma \leq 2^K} \sigma^{-2} \|f\|_2^2 \nonumber
 \lesssim 2^{4k}\|f\|_2^2.
\end{equation}
The support of each $\Hh_{\S_\sigma^{(k)}}^\star$, $1\leq \sigma\leq 2^K$ is  contained in $A^1$, so that
\begin{equation}
\label{outsideA1} \Hh_{\S_\sigma^{(k)}}^\star f(x) \leq \sum_{\sigma \leq 1} \Hh_{\S_\sigma^{(k)}}^\star f(x)  \leq \sum_{\sigma \leq 1} \sum_{\t \in \F_\sigma} \Hh_{\t^{(k)}} f (x), \qquad x \not\in A^1.
\end{equation}
We now construct the second component of $A$. We begin by excising an exceptional set for each tree $\t \in \F_\sigma$ appearing in the sum \eqref{outsideA1}: setting
\begin{equation}
\label{eset-tree}
A^{2,\size}_{\t,\sigma} = \{x \in \sh_k(\t): \Hh^\star_{\t^{(k)}} f (x) \gtrsim 2^{2k} \sigma|10\log\sigma|^2 (\log \log N)^2 \},\qquad  \t \in \F_{\sigma}
\end{equation}
we will show that
\begin{equation}
\label{eset-treeest}
|A^{2,\size}_{\t,\sigma}| \lesssim \sigma^{10}2^{2k}|\sh(\t)|.
\end{equation}
Set $A^{2,\size}_\sigma:= \cup_{\t \in \F_\sigma} A^{2,\size}_{\t,\sigma}$. Using \eqref{bessel-prel}, we have the estimate  \begin{equation}
\label{eset-treesum}
 |A^{2,\size}_\sigma| \leq \sum_{\t \in \F_\sigma} |A^{2,\size}_{\t,\sigma}| \leq 2^{2k}\sigma^{10}\sum_{\t \in \F_{\sigma}} |\sh(\t)| \lesssim 2^{2k}\sigma^{-8}\|f\|_2^2
\end{equation}
We also excise an exceptional set coming from the whole forest $\S_\sigma$:
\begin{equation}
\label{eset-count}
A^{2,\mathrm{count}}_\sigma:= \Big\{\sum_{\t \in \F_\sigma} \cic{1}_{\sh_k( \t)} (x) \gtrsim  \sigma^{-2}|\log \sigma|^2 \Big\}.
 \end{equation}
Another application of \eqref{bessel-prel} entails
\begin{equation} \label{eset-countest}
|A^{2,\mathrm{count}}_\sigma| \leq \frac{\sigma^{-2}}{|\log \sigma|^2 }\sum_{\t \in \F_\sigma}|\sh_k(\t)| \lesssim \frac{2^{4k}}{|\log \sigma|^2} \|f\|_2^2.
\end{equation}
Putting all pieces together, we obtain our second component $A^2$:
\begin{equation} \label{eset-A2}
A^2= \bigcup_{\sigma \lesssim 1} \big( A^{2,\size}_\sigma\cup A^{2,\mathrm{count}}_\sigma\big), \qquad |A^2| \lesssim  2^{4k}  \|f\|_2^2.
\end{equation}
the estimate on the measure of $A^2$ is obtained by summing over $\sigma\lesssim 1$ the estimates in \eqref{eset-treesum}  and \eqref{eset-countest}. Thus, the first part of \eqref{sect63-1*} holds true.

We are left to show the second part of \eqref{sect63-1*}. Recall that   $\mathsf{crown}(\t)= \bigcup\{\omega=\omega_{2s}: s \in \t\}$. A a  consequence of the definition of $A^{2,\mathrm{count}}_\sigma$ and of property \eqref{count-prel}, we have that
\begin{equation}
\label{count-final} \sup_{x \not\in A} \sup_{v \in \V}
\sum_{\t \in \F_\sigma } \cic{1}_{\sh_k(\t)}(x) \cic{1}_{\mathsf{crown}(\t)} (v) \lesssim \min \{\log N, \sigma^{-2} |\log \sigma|^2 \}.
\end{equation}
Then, for $x \not \in A$, choose  $v (x) \in \V$  which attains the supremum in $ \Hh_{\S_\sigma^{(k)}}^\star $. We have
\begin{align*}
\label{}
\Hh_{\S_\sigma^{(k)}}^\star f(x) &=    \Big| \sum_{\t \in \F_\sigma} \sum_{s \in \t} \l f,\varphi_s \r  \psi_{s}^{(k)}(x)\cic{1}_{  \omega_{2s}} (v(x)) \Big|\\ & = \Big| \sum_{\t \in \F_\sigma} \cic{1}_{\sh_k(\t)}(x) \cic{1}_{\mathsf{crown}(\t)}(v(x)) \sum_{s \in \t} \l f,\varphi_s \r  \psi_{s}^{(k)}(x)\cic{1}_{v(x) \in \omega_{2s}}\Big|  \nonumber \\
& \lesssim \min\{\log N, \sigma^{-2} |\log \sigma|^2 \} \Big(\sup_{\t \in \F_\sigma} \Hh_{\t^{(k)}}^\star f(x)\Big) \\ & \lesssim 2^{2k}(\log\log N)^2\min\{\log N \sigma |\log \sigma|^2 , \sigma^{-1} |\log \sigma|^4 \}
\end{align*}
Taking advantage of \eqref{outsideA1}, we then get, for $x \not \in A$,
 \begin{align*}\Hh_{\S_ {(k)}}^\star f(x)  &\leq \sum_{\sigma \lesssim 1} \Hh_{\S_\sigma^{(k)}}^\star
f(x)   \\& \lesssim 2^{2k}(\log\log N)^2\sum_{\sqrt{\log N} \geq \sigma^{-1} \geq 1} \sigma^{-1} |\log \sigma|^4 + 2^{2k}\log N(\log\log N)^2\sum_{\sigma \lesssim\sqrt{\log N}} \sigma |\log \sigma|^2  \\ &\lesssim 2^{2k} \sqrt{\log N}(\log\log N)^6,
\end{align*}
which completes the proof of the second part of \eqref{sect63-1*}. This in turn proves the weak-type bound \eqref{sect63-1}, and finally, Theorem \ref{main-ub3}. In the remainder of the section, we will show how to achieve a decomposition $\S=\cup_{\sigma \leq 2^K} \S_\sigma$ complying with properties \eqref{bessel-prel}, \eqref{count-prel}, and whose trees satisfy the exceptional set estimate \eqref{eset-treeest}.
\subsection{Lacunary and overlapping parts of $\S$} First of all, it suffices to assume that $\#(\V \cap \omega_{2s}) \geq 1$ for each $s \in \S$. Indeed, tiles with empty $\V \cap \omega_{2s}$ do not contribute to the model sum $\Hh_{\S^{(k)}}^\star.$
We  decompose
$$
\S = \L \cup \O, \quad \L=\{s \in \S:\V \cap \omega_{1s}=\emptyset\}, \quad  \O=\{s \in \S:\V \cap \omega_{1s}\neq\emptyset\}
$$
Then, $\S_\sigma$ will be obtained as the union   $\L_\sigma \cup \O_\sigma$. The forest $\L_\sigma$ will be the union of lacunary trees $\t \in \F_\sigma^\L$, while the forest $\O_\sigma$ will be the union of overlapping trees $\t \in \F_\sigma^\O$.  For $* \in \{\L,\O\}$, and $\sigma \leq 2^K$,   we will   achieve the properties
\begin{align}
\label{bessel-in} &
\sum_{\t \in \F^*_{\sigma}} |\sh(\t)| \lesssim \sigma^{-2}\|f\|_2^2,\\
\label{count-in} &
\sup_{v \in \V}\sum_{\t \in \F_\sigma^*} \cic{1}_{\mathsf{crown}(\t)} (v) \leq C\log N,  \\
\label{eset-in}& \t \in \F^*_\sigma \implies
\big|\big\{x : \Hh^\star_{\t^{(k)}} f (x) \gtrsim 2^{2k} \sigma|10\log\sigma|^2 (\log \log N)^2 \big\}\big| \lesssim \sigma^{10}2^{2k}|\sh (\t)|.
\end{align}
This will imply the corresponding properties for $\S_\sigma$, and in turn, complete the proof of Theorem 3.
\subsection{The lacunary part $\L$} We begin by noting that the intervals
$
\{\omega=\omega_{1s}: s \in \L\}
$
are pairwise disjoint. Indeed, if we had $\omega_{1s}$ strictly contained in $ \omega_{1s'}$, for $s, s' \in \L$, then $\omega_{2s} \subset \omega_{1s'}$ too. This would force $\#(\V \cap \omega_{1s'}) \geq 1$, which we ruled out. As a consequence, all trees $\t \subset \L$ are of lacunary type, and the almost-orthogonality relation
\begin{equation}
\label{ao}
\sum_{s \in \L } |\l f, \varphi_s \r|^2 \lesssim \|f\|_2^2
\end{equation}
holds. We organize $\L$ into  subcollections $\L(j)$ by setting
$$
\L(j)= \{s \in \L: 2^{j} \leq \omega_{2s} <2^{j+1}\}, \qquad j=1, \ldots, \log N.
$$
Recall that $\pi_s =\omega_{1s} \cup \omega_{2s}$. For each $\L(j),$ we distinguish the minimal (with respect to inclusion) dyadic intervals $\{\pi^j_\ell\}$ inside $\{\pi=\pi_s: s \in \L(j)\}$, and form the lacunary trees
$$
\t(j,\ell)= \{s \in \L(j): \omega_{2s} \supset \pi^j_\ell\}, \qquad \L(j) = \cup_\ell \t(j,\ell).$$ It is clear that each $\pi_{s}$ contains exactly one $\pi^j_\ell$. This forces  $\pi_s\cap \pi_{s'}=\emptyset$ whenever $s \in \t(j,\ell), s' \in \t(j,\ell'), \ell'\neq \ell.$ Thus for each $j$ the sets $\{\mathsf{crown}(\t(j,\ell)):\ell\}$ are pairwise disjoint. We conclude that
\begin{equation}
\label{count-lacunary1}
\sup_{v \in S^1}\sum_{j=1}^{\log N} \sum_{ \ell } \cic{1}_{\mathsf{crown}(\t(j,\ell))}(v) \leq \log N.
\end{equation}
We now construct $\F_{\sigma}^\L=\cup_{j=1}^{\log N }\cup_ \ell \t_\sigma(j,\ell)$ by    chopping each  tree $\t (j,\ell)$ into subtrees with
$\size(\t_\sigma(j,\ell))\leq \sigma$. The chopping is obtained via the following iterative procedure.
\begin{itemize}
\item \textsf{INIT} $\mathrm{Stock}:=\t (j,\ell)$, $\sigma=2^K$;
\item \textsf{WHILE} $\mathrm{Stock}\neq \emptyset$, select the largest collection $\t_\sigma(j,\ell)\subset \mathrm{Stock}
$ with
$$
\sum_{s \in \t_\sigma(\alpha) } |\l f,\varphi_s\r|^2  \geq{\textstyle\frac{ \sigma^{2}}{4}} |\sh(\t_\sigma(j,\ell))|.
$$
Set $\mathrm{Stock}:=\mathrm{Stock}\backslash \t_\sigma(j,\ell)$; $\sigma:=\frac\sigma2$.
\end{itemize}  The almost-orthogonality property \eqref{ao} yields that\begin{equation} \label{sect2-orthpf4v}
 \sum_{j=1}^{\log N} \sum_{\ell}|\sh(\t_\sigma(j,\ell))|\lesssim \sigma^{-2}\sum_{s \in \L } |\l f,\varphi_s\r|^2 \lesssim \sigma^{-2 } \|  f\|_2^2,
\end{equation}
so that \eqref{bessel-in} for $*=\L$ is verified. Moreover, \eqref{count-in} follows \emph{a fortiori} from \eqref{count-lacunary1}
All that is left to verify is the single tree estimate \eqref{eset-in}. Let $\t= \t(j,\ell) \in \F^\L_\sigma$ ($j$ and $\ell$ have no particular relevance). For each $v \in \V \cap \mathsf{crown}(\t)$, let us isolate the subtree of $\t$ contributing at $v$, that is
$
\t(v)= \{s \in \t:v \in \omega_{2s}\}.
$ Set $\T(\t)=\{\t(v): v  \in \V \cap \mathsf{crown}(\t)\}$.
With this position, using the notation of \eqref{st-lac-k-ms}, we have
$$\Hh_{\t^{(k)}}^\star f(x)= \sup_{\t' \in \T(\t)} |f_{\t'}^{(k)}(x)|.
$$
We claim that, as a consequence of the Vargas set property, $\T(\t)$ contains at most $C\log N$ distinct subtrees $\t'$. The proof of the claim is postponed at the end of the section. Assuming for now that $\#\T(\t) \leq C\log N$, we can complete the proof of the single tree estimate as follows:
\begin{align}
\label{inequalcraz} &\quad
\big|\big\{x : \Hh^\star_{\t^{(k)}} f (x) \gtrsim 2^{2k} \sigma|10\log\sigma|^2 (\log \log N)^2 \big\}\big| \\\nonumber & \leq \sum_{\t' \in \T(\t)} \big|\big\{x : |f_{\t'}^{(k)}(x)| \gtrsim 2^{2k} \sigma|10\log\sigma|^2 (\log \log N)^2 \big\}\big|
\\ & \leq \sum_{\t' \in \T(\t)} \exp\Big(-c \textstyle \sqrt{\frac{2^{2k}\sigma |\log\sigma|^2 (\log \log N)^2}{\size(\t')}} \Big)|\sh_k(\t')| \nonumber\\&\lesssim 2^{2k}\sigma^{10}  \sum_{\t' \in \T(\t)} \e^{ -\log \log N} |\sh_k(\t')| \leq 2^{4k}\sigma^{10} |\sh(\t)|. \nonumber
\end{align}
We used Lemma \ref{st-lac-k}, and the fact that $\size(\t') \leq \sigma$. Thus, the proof of \eqref{eset-in}, for $\star=\L$ is complete.

\begin{proof}[Proof of the claim]
Ordering  $ \V \cap \mathsf{crown}(\t)=\{v_1,v_2,\ldots\}$ counterclockwise, we say that $v_i$ is a cutoff direction if $\t(v_{i})\supsetneq \t(v_{i-1})$. We show that the sequence $\{v_i: v_i$ is a cutoff$\}$ is lacunary with a top direction of $\t$, $v_\t$, as its node, and thus it can have at most $C\log N$ elements, which proves the claim. By pigeonholing, it suffices to prove this under the additional sparseness assumption
$$
s, s' \in \t, \, \omega_{2s} \subset \omega_{2s'} \implies 8|\omega_{2s}| \leq  |\omega_{2s'}|.
$$
This implies that $|\mathsf{crown}(\t(v_{i_m}))|\geq 8|\mathsf{crown}(\t(v_{i_{m+1}}))|$ for two consecutive cutoffs $v_{i_{m}}$ and $v_{i_{m+1}}$ with      $i_m<i_{m+1}$.
Thus
$$
|v_{i_{m}}-v_\t| \geq \frac12|\mathsf{crown}(\t(v_{i_m}))| \geq 4 |\mathsf{crown}(\t(v_{i_{m+1}}))| \geq 2|v_{i_{m+1}}-v_\t|,
$$
and this shows that the cutoffs are lacunary with node $v_\t$. The claim is proved.
\end{proof}
\subsection{The overlapping part $\O$}
We start with the construction of $\F_\sigma^\O$. The trees of $\F_\sigma^\O$ will be overlapping trees obtained as saturation of conical trees (see the definitions in Section 6) with conical size roughly $\sigma$. Let us give a precise definition of saturation.

Let $\ct$ be a conical tree. Note that $\{R_s:s \in \ct \}$ is a collection of dyadic rectangles with the same shape, coming from a single dyadic grid. Thus, the tiles
$\ct_{\max} $ with $R_s$  maximal  are well defined and   $\sh(\ct)$ is the disjoint union of $\{R_s:s \in \ct_{\max} \}$.
The saturation $\t(\ct)$ of a conical tree is defined as
$$
\t(\ct)=\{s' \in \S: \omega_{1s'} \supset \omega_\ct, R_{s'} \subset 10 R_{s} \textrm{ for some } s \in \ct_{\max}\}.
$$ We omit the dependence on $\ct$ on the notation when no confusion arises.
Clearly, $\t$ is an overlapping tree with top $\omega_{\ct}$ and $|\sh(\t)| \leq 10 |\sh( \ct)| .$

The next lemma,  along the lines of (for example) \cite[Lemma 8.2]{THWPA}  or \cite[Lemma 6.6]{MTT} will be used in the iterative construction of $\F_\sigma^\O$. We sketch a proof at the end of the subsection.
\begin{lemma}[Size Lemma] \label{sect6-sizelemma}
Let $\O'\subset\O$ be a collection of tiles with $\size_\nabla(\O') =\sigma$.
Then $\O':=\O'_{\mathrm{small}} \cup \O'_{\mathrm{big}}$ where
\begin{equation}
 \mathrm{size}_\nabla (\O'_{\mathrm{small}}) \leq \frac{\sigma}{2}
\end{equation}
and
  $\O'_{\mathrm{big}} = \cup_{\t \in \F} \t$.  Each $\t \in \F$ is a saturation of a conical tree $\ct$, and
\begin{equation}
 \sum_{\t \in \F} |\sh(\t)| \lesssim  \sigma^{-2 }\| f\|_2^2; \label{besself}
 \end{equation}
\end{lemma}
Starting with $\O'=\O$ (which has $\size_\nabla(\O)=2^K$) and iteratively applying Lemma \ref{sect6-sizelemma}, we obtain the decomposition
$
\O= \cup_{\sigma\leq 2^K}   \O_\sigma,  $
with
\begin{align}   \label{conicalsizeppa}    & \size(\O_\sigma)\leq \sigma,\\
\label{conicalsizeppb}    &\displaystyle \O_{\sigma}=\bigcup_{\t \in \F_\sigma^{\O}} \t,  \textrm{\emph{ and each }}  \t \in \F_\sigma \textrm{\emph{ is the saturation of a conical tree,}}\\
 \label{conicalsizeppc}  &  \# \big( \cic{\omega}_\t:= \{\omega:\omega=\omega_{2s}  \textrm{\emph{ for some }} s \in \tau(\ct)\} \big) \leq C\log N,\\
\label{conicalsizeppd}  \displaystyle& |\sh(\t)| \lesssim  \sigma^{-2 }\| f\|_2^2.
\end{align}
The only property which is not straightforward from the lemma is \eqref{conicalsizeppc}. This is an easy consequence of the Vargas set property: the intervals of $\cic{\omega}_\t$ are pairwise disjoint, and each contains at least one $v\in \V$. Denote with  $v_\ell$ any of the directions of $\V$ contained in $\omega \in \cic{\omega}_\t$, $|\omega|=2^{-\ell}$. Then $2^{-\ell-1 }\leq |v_\ell-v_\t| \leq 2^{-\ell}$, which means $\{v_\ell\}$ is a lacunary sequence with node $v_\t$. Thus, $\{v_\ell\}$ (and in turn $\cic{\omega}_\t$) contains at most $C\log N$ elements.

We now reshuffle the trees in $\F_\sigma^\O$ to make sure that \begin{equation}
\label{count-ov1}
\sup_{v \in S^1}\sum_{\t \in \F_\sigma^\O}   \cic{1}_{\mathsf{crown}(\t)}(v) \leq C \log N
\end{equation}
also holds true.
 For each dyadic interval $\omega$, consider the collection of tiles $s \in \F_\sigma^\O$ with $\omega_{2s}=\omega$. These tiles are partially ordered by the relation $s\ll s'$ if $R_s \subset R_s'$. Then remove each tile $s$ which is not maximal with respect to this order relation from the   tree $\t$ to which $s$ belongs, and assign $s$ to the   tree $\t'$ where the maximal tile $s'$ s.t.\ $s\ll s'$ belongs. We keep calling the resulting forest $\F_\sigma^\O$. Since this reshuffling does not modify  $\O_\sigma$ and since it only makes  $ \cic{\omega}_\t$ and $\sh(\t)$ smaller, properties \eqref{conicalsizeppa}-\eqref{conicalsizeppd} still hold.

 As a result of the reshuffling, all the tiles $s \in \O_\sigma$ with $\omega_{2s}=\omega$ will belong to the same  tree $\t$, and \eqref{count-ov1} follows easily. Suppose that the sum in \eqref{count-ov1} is equal to $r>0$. Then $v$ belongs to $r$ (nested) dyadic intervals $\cic{\omega}=\{\omega\}$ and  for each of them we have $\omega=\omega_{2s}$ for  $s$ in some (uniquely determined) tree $\t$. Then, the collection $\{\omega_{1s}: \omega_{2s}=\omega, \omega \in \cic{\omega}\}$ is pairwise disjoint and each $\omega_{1s}$ contains at least one element of $  \V$. The same argument used to show \eqref{conicalsizeppc} implies that $r \leq C\log N$.

Having at our disposal \eqref{besself} and \eqref{count-ov1}, we have established \eqref{bessel-in} and \eqref{count-in} for $*=\O$. We are left with proving the single tree estimate \eqref{eset-in} for an overlapping tree $\t \in \F_\sigma$.
We observe that $\t$ is a disjoint union of (at most $C\log N$) conical trees, $$
\t = \cup_{\omega\in \cic{\omega}_\t} \ct(\omega), \qquad
\ct(\omega):=\{s \in \t:\omega_{2s}=\omega\}.
$$
The intervals in $\cic{\omega}_\t$ are pairwise disjoint, so that, with the notation of \eqref{st-ov-k-est},
$$\Hh_{\t^{(k)}}^\star f(x)= \sup_{\omega\in \cic{\omega}_\t} |f_{\ct(\omega)}^{(k)}(x)|.
$$
Then, the estimate
$$
\big|\big\{x : \Hh^\star_{\t^{(k)}} f (x) \gtrsim 2^{2k} \sigma|10\log\sigma| \log \log N  \big\}\big| \lesssim \sigma^{10}2^{2k}|\sh (\t)|,$$
which is even stronger than what we need for \eqref{eset-in}, is obtained in the same way as  \eqref{inequalcraz}, applying the estimate for conical trees \eqref{st-ov-k-est} in place of \eqref{st-lac-k-est}. This concludes the proof of the theorem.
\begin{proof}[Proof of Lemma \ref{sect6-sizelemma}] We insist only on the points which are different from the classical size lemma in \cite{LT}; see \cite{THWPA} and \cite{MTT} for more details.

We can assume that the uncertainty intervals $\{\omega_{1s}: s \in \O'\}$ come from a single 2-sparse dyadic grid.
The following iterative procedure is used to separate $\O'_{\mathrm{small}}$ from $\O'_{\mathrm{big}}$:
\begin{itemize}
\item \textsf{INIT}  $\F_+,\F'=\emptyset,\; \mathrm{Stock}:=\O',\; \O'_{\mathrm{big}}:=\emptyset;$
 \item   \textsf{WHILE}  $\size_{\nabla}(\mathrm{Stock}:)\geq \frac{\sigma}{2}$
\begin{itemize} \item[$\cdot$] select a conical   $\ct$ with $ \displaystyle
\sum_{s \in \ct} |\langle f,\varphi_s\rangle|^2 \geq \textstyle \frac{\sigma^2}{4}|\sh(\ct) |$, with $\omega_\t$  minimal, and   $\sh(\t)$ maximal (with respect to inclusion);
\item[$\cdot$] \textsf{UPDATE}   $\F_+=\F_+ \cup \ct$, $\F=\F'\cup\t(\ct)$, $\; \mathrm{Stock}:= \mathrm{Stock}-\t(\ct).$
\end{itemize}
\end{itemize}
When the algorithm stops, set $\O'_{\mathrm{small}}:= \mathrm{Stock} $, $\O'_{\mathrm{big}}=\cup_{\t \in \F } \t $, $\O'_+=\cup_{\ct \in \F_+} \ct.$  The iteration procedure ensures that $\size_\nabla(\O'_{\mathrm{small}}) \leq \frac\sigma 2$. We are left with  proving the inequality
\begin{equation}
\label{bessel}
\sigma^2\sum_{ \ct \in \F_+} |\sh(\ct)| \lesssim \|f\|_2^2.
\end{equation}
We will use the notation $F(\S'):=\sum_{s \in \S'} \l f,\varphi_s\r \varphi_s.$ Standard arguments reduce the proof of \eqref{bessel} to showing that \begin{equation}
\label{bessel-pf1}
\|F(\O'_+)\|_2^2 \lesssim \sigma^2 \sum_{\ct \in \F_+} |\sh(\ct)|.
\end{equation}
For   $s \in \ct \in \F_+$, define $\mathbf{B}(s)= \{s' \in \O'_+ -\ct:  \omega_{1s}=\omega_{1s'} \}, \, \mathbf{C}(s)= \{s' \in  \O'_+ -\t:  \omega_{1s}\subsetneq{\omega}_{1s'} \}. $ The tiles $ s' \in \mathbf{B}(s)\cup \mathbf{C}(s)$ are the ones with $\l\varphi_s,\varphi_{s'} \r \neq 0$. Then
$$
\|F(\O'_+)\|_2^2 \lesssim  \sum_{\ct \in \F_+} \|F(\ct)\|^2_2
+
\sum_{s \in \O'_+} \l f,\varphi_s\r \l\varphi_s, F(\mathbf{B}(s)) \r + \sum_{s \in \O'_+} \l f,\varphi_s\r \l\varphi_s, F(\mathbf{C}(s)) \r
$$
The fact that first and second summand obey the bound  of \eqref{bessel-pf1} is a simple consequence of the almost-orthogonality of the wave packets inside a single conical tree.  The third summand is controlled   by summing over   $\ct \in \F_+$ the estimate
\begin{equation*}
\sum_{s \in \ct} \l f,\varphi_s\r \l\varphi_s, F(\mathbf{C}(s)) \r  \lesssim \sigma^2 |\sh(\ct)|.
\end{equation*}
 Clearly this follows if we establish that, for each $\bar s \in \ct_{\max}$,
\begin{equation}
\label{bessel-pf3}
\sum_{s \in \ct(\tilde s) } \l f,\varphi_s\r \l\varphi_s, F(\mathbf{C}(s)) \r  \lesssim \sigma^2 |R_{\bar s}|.
\end{equation}
where $\ct(s):=\{s \in \ct: R_s \subset R_{\tilde s}\}.$
Split the collection $\ct(\bar s)$ into $$\ct(\bar s)=\bigcup_{n\geq 0}  \ct(\bar s)^{(n)},\qquad  \ct(\bar s)^{(n)}: =\{s \in \ct(\bar s): 2^{n}R_s \subset 5R_{ \bar s}, 2^{n+1}R_s \not\subset 10R_{\bar s} \}  .$$
The rectangles  $\{R_s: s \in \ct(\bar s) \}$ come from a single dyadic grid in the plane, hence
$$
\sum_{s \in  \ct(\bar s)^{(n)} } |R_{  s}| \lesssim 2^{  n}  |R_{\bar s}|,
$$
so that \eqref{bessel-pf3} is obtained by summing over $s \in  \ct(\bar s)^{(n)}$ and $n\geq 0$ the estimate
  \begin{equation}
 \label{bessel-pf4}
 \sum_{s' \in \mathbf{C}(s)} |\l f,\varphi_s\r| |\l \varphi_{s'}, f\r| |\l\varphi_s, \varphi_{s'} \r|  \lesssim 2^{-10n}\sigma^2 |R_{s}|.
\end{equation}
We now show \eqref{bessel-pf4}, beginning with two  observations.

 \noindent \textsc{Obs.1}: let $s \in \ct$ and $s' \in \mathbf{C}(s)$. The 2-sparseness property we assumed implies $\omega_{\ct} \subset \omega_{1s'}$, so that   $\ct$ had been  selected before the tree containing $s'$ (by minimality of $\omega_\ct$).   If $\bar s \in \ct_{\max}$ is the one tile for which $s \in \ct(\bar s)$,  the saturation procedure then forces    $8R_{\bar s} \cap R_{s'} = \emptyset.$  \vskip1.5mm
\noindent \textsc{Obs.2}: by the same token, let $s  \ct$, $s' \neq s''$ with $s',s'' \not\in \ct $. Then
$$\omega_{1\ct} \subset   \omega_{1s'} \cap \omega_{1s''},   \,
\annf(s') =\annf(s'')\implies
R_{s'}\cap R_{s''}=\emptyset. $$
Indeed,  the 2-sparseness ensures
$\omega_{\ct} \subset \omega_{1\ct'}\cap  \omega_{1\ct''}$.  In the case $\omega_{1\ct'}=\omega_{1\ct''}$, the conclusion follows because $R_{s'}$ and $R_{s''}$ are dyadic rectangles with the same orientation and sidelengths. If not,  assume with no loss in generality that $\omega_{\ct'} \subsetneq \omega_{1\ct''} $, then the conclusion follows exactly as in Observation 1. \vskip1.5mm

Therefore the rectangles $\{R_{s'}: s' \in \mathbf{C}(s) \}$ are pairwise disjoint (by Observation 2) and do not intersect $2^{n} R_s$, (by Observation 1), which is strictly contained in $8R_{\bar s}.$
Since $\size_\nabla$ controls the size of the single tiles, $ |\l f,\varphi_s\r| \lesssim \sigma \sqrt{|R_s|},$ $ |\l f,\varphi_{s'}\r| \lesssim \sigma \sqrt{|R_{s'}|}.$ Using this fact and the rapid decay of the wave packets,  \eqref{bessel-pf4} is a consequence of
$$
 \sum_{s' \in \mathbf{C}(s)} |R_{s'}| \Big(1+\frac{\mathrm{dist}(R_s,R_{s'})}{|R_s|}\Big)^{-100} \lesssim \int_{(2^{n}R_s)^c}   \Big(1+\frac{\mathrm{dist}(x,R_{s'})}{|R_s|}\Big)^{-100} \,\d x\lesssim 2^{-10n} |R_{s}|.
$$
This concludes the proof of \eqref{bessel-pf4}, and, in turn, of Lemma \ref{sect6-sizelemma}
\end{proof}

\section{Final remarks}Let $f:\R^2\to\mathbb C$.
The main unresolved issue in the paper is whether the model sum operator
$$\mathsf{MS}f(x)=\sup_{v\in \V}|\sum_{s\in\S_u } {\l f, \varphi_{s} \r} \psi_{s}(x)  1_{\omega_{2s}}(v)|$$
has the $L^2\to L^{2,\infty}$ norm of order $O(\sqrt{\log N}(\log\log N)^{O
(1)})$, for arbitrary $\V\subset S^1$ with $N$ elements. Here, as before, $\varphi_{s},\psi_s$ are product wave packets adapted to $R_s$.

Consider the related operators
$$\mathsf{S}f(x)=\sup_{v\in\V}\bigg(\sum_{\ann}\big|\sum_{s\in\S_u:\ann(s)=\ann } {\l f, \varphi_{s} \r} \psi_{s} (x)1_{\omega_{2s}}(v)\big|^{2}\bigg)^{1/2}$$
$$\mathsf{SQ}f(x)=\sup_{v\in\V}\bigg(\sum_{s\in\S_u} |{\l f, \varphi_{s} \r}|^{2}\frac{1_{R_s}(x)}{|R_s|}1_{\omega_{2s}}(v)\bigg)^{1/2}$$
$$\mathsf{SC}f(x)=\sup_{v\in\V}\sup_{\omega\subset S^1}\bigg(\sum_{s\in\S_u:\omega_s=\omega} |{\l f, \varphi_{s} \r}|^{2}\frac{1_{R_s}(x)}{|R_s|}1_{\omega_{2s}}(v)\bigg)^{1/2}.$$
Note the striking simplicity of $\mathsf{SC}$. For each $x$, the contribution comes from tiles with a fixed eccentricity (i.e.\ from a single frequency cone).

By invoking a simple randomization argument, it easily follows that
$$\|\mathsf{SC}\|_{L^2\to L^{2,\infty}}\lesssim\|\mathsf{SQ}\|_{L^2\to L^{2,\infty}}\lesssim\|\mathsf{S}\|_{L^2\to L^{2,\infty}}\lesssim\|\mathsf{MS}\|_{L^2\to L^{2,\infty}}.$$
The fact that $\|\mathsf{S}\|_{L^2\to L^2}\lesssim \sqrt{\log N}$ was proved in \cite{CD10}, and this is optimal for generic $\V$. However, we strongly suspect that $\|\mathsf{S}\|_{L^2\to L^{2,\infty}}\lesssim (\log \log N)^{O(1)}$. If one proves this, then the conjectured bound $\|\mathsf{MS}\|_{L^2\to L^{2,\infty}}\lesssim \sqrt{\log N}(\log\log N)^{O
(1)}$ follows via the Chang-Wilson-Wolff inequality, as described in Section \ref{lac}. The methods in this paper seem to be unable to conclude even the bound $\|\mathsf{SC}\|_{L^2\to L^{2,\infty}}\lesssim (\log \log N)^{O(1)}$.

\bibliography{biblio}{}
\bibliographystyle{amsplain}

 \end{document}